\documentclass[10pt,letterpaper,twoside]{amsart}
\usepackage[left=1.2in,right=1.2in,top=1in,bottom=1in,headheight=12pt,headsep=18pt]{geometry}
\usepackage{amsmath,amssymb,mathtools,mathrsfs,microtype}
\usepackage{fancyhdr,xcolor,tikz,array,longtable,needspace}
\definecolor{citeblue}{RGB}{50,155,215}
\usepackage[colorlinks=true,linkcolor=blue,citecolor=citeblue,urlcolor=blue]{hyperref}
\newtheorem{theorem}{Theorem}[section]
\newtheorem{lemma}[theorem]{Lemma}
\newtheorem{proposition}[theorem]{Proposition}

\theoremstyle{definition}
\newtheorem{definition}[theorem]{Definition}
\AtBeginDocument{%
\setlength{\abovedisplayskip}{4pt plus 1pt minus 1pt}%
\setlength{\belowdisplayskip}{4pt plus 1pt minus 1pt}%
\setlength{\abovedisplayshortskip}{3pt plus 1pt minus 1pt}%
\setlength{\belowdisplayshortskip}{3pt plus 1pt minus 1pt}}
\title{Tuza's Conjecture for Graphs of\\Maximum Degree at Most 10}
\author{Amrit Kandasamy}
\date{}
\hypersetup{pdftitle={Tuza's Conjecture for Graphs of Maximum Degree at Most 10}}

\tikzset{proof vertex/.style={circle,draw=black,fill=white,inner sep=0.8pt,font=\scriptsize,minimum size=3.3mm},
leaf one/.style={violet!85!black,line width=0.85pt},
leaf two/.style={teal!85!black,line width=0.85pt,densely dashed},
leaf three/.style={orange!85!black,line width=1pt,densely dotted,line cap=round},
leaf four/.style={brown!85!black,line width=0.85pt,dash dot}}
\newcommand{\PackingRow}[2]{%
\begin{tikzpicture}[baseline=-0.5ex,x=1cm,y=1cm]
\node[anchor=east,font=\small] at (-0.32,0) {$#1$};
\foreach \PackingFrom/\PackingTo [count=\PackingIndex from 0] in {#2}{%
\draw[red!75!black,line width=0.7pt] (1.17*\PackingIndex,0)--(1.17*\PackingIndex+0.58,0);
\node[proof vertex] at (1.17*\PackingIndex,0) {$\PackingFrom$};
\node[proof vertex] at (1.17*\PackingIndex+0.58,0) {$\PackingTo$};}
\end{tikzpicture}}
\newcommand{\LeafPacking}[5]{%
\begin{tikzpicture}[baseline=-0.5ex,x=1cm,y=1cm]
\foreach \PackingVertex/\PackingX/\PackingY in {u/0.00000/1.25000,s/0.98995/0.88388,v/1.40000/0.00000,d/0.98995/-0.88388,c/0.00000/-1.25000,r/-0.98995/-0.88388,a/-1.40000/-0.00000,b/-0.98995/0.88388}{\coordinate (\PackingVertex) at (\PackingX,\PackingY);}
\foreach \PackingFrom/\PackingTo in {#1}{\draw[leaf one] (\PackingFrom)--(\PackingTo);}
\foreach \PackingFrom/\PackingTo in {#2}{\draw[leaf two] (\PackingFrom)--(\PackingTo);}
\foreach \PackingFrom/\PackingTo in {#3}{\draw[leaf three] (\PackingFrom)--(\PackingTo);}
\foreach \PackingFrom/\PackingTo in {#4}{\draw[leaf four] (\PackingFrom)--(\PackingTo);}
\foreach \PackingFirst/\PackingSecond/\PackingThird in {#5}{%
\draw[red!75!black,line width=0.9pt] (\PackingFirst)--(\PackingSecond)--(\PackingThird)--cycle;}
\foreach \PackingVertex in {u,v,r,s,a,b,c,d}{\node[proof vertex,font=\footnotesize,minimum size=3.8mm] at (\PackingVertex) {$\PackingVertex$};}
\end{tikzpicture}}

\begin{document}
\begin{abstract}
For a finite simple graph $G$, let $\nu(G)$ be the maximum number of pairwise edge disjoint triangles and $\tau(G)$ the minimum number of edges meeting every triangle. Tuza conjectured in 1981 that every such $G$ satisfies $\tau(G) \le 2 \nu(G).$ We prove the conjecture for the class of graphs with $\Delta\le10.$
\end{abstract}
\maketitle
\thispagestyle{plain}
\begin{center}
\small Code and verification for Lemma \ref{subcubic-basis-ten} and Proposition \ref{forty-ten} \\[0.15em]
\href{https://github.com/Amkan2026/tuza-degree-ten}{https://github.com/Amkan2026/tuza-degree-ten}
\end{center}
\section{Introduction}
For a finite simple graph $G$, let $\nu(G)$ be the maximum number of pairwise edge disjoint triangles and $\tau(G)$ the minimum number of edges meeting every triangle.
Tuza \cite{Tuza} conjectured that $\tau(G)\le2\nu(G)$.
Puleo \cite{Puleo} proved this when every subgraph has average degree less than $7$, and Gupta \cite{Gupta} proved it for $\Delta(G)\le7$.
\leavevmode\vspace{0.5em} \\
The edges of a maximum packing meet every triangle, so $\tau(G)\le3\nu(G)$.
Tuza's bound is sharp since $\nu(K_4)=1$ and $\tau(K_4)=2$.
\begin{theorem}\label{main}
Every finite simple graph $G$ with $\Delta(G)\le10$ satisfies
\[
\tau(G)\le2\nu(G).
\]
\end{theorem}
\section{Preliminaries}\label{reductions-ten}
Write $N[x]=N(x)\cup\{x\}$, $E(S)$ for the edges of a packing $S$, and $E(A,B)$ for the edges with one endpoint in each of $A,B$.
The link at $x$ is $G[N(x)]$, and the codegree of $xy$ is $|N(x)\cap N(y)|$.
\leavevmode\vspace{0.5em} \\
A reduction removes triangles at a cost of at most two cover edges per packed triangle.
It also needs to leave the new packing disjoint from every packing in the smaller graph.
\begin{definition}[Puleo {\cite[Definition 2.1]{Puleo}}]\label{reducing-ten}
A nonempty vertex set $U$ is reducible if a triangle packing $S$ and an edge set $X$ satisfy
$$
|X|\le2|S|$$
$$G-X\text{ has no triangle meeting }U$$
$$E(S)\setminus E(U,V(G))\subseteq X.$$
\end{definition}
\begin{lemma}[Puleo {\cite[Lemma 2.2]{Puleo}}]\label{combine-ten}
If $U$ is reducible and $\tau(G')\le2\nu(G')$ for $G'=(G-X)-U$, then $\tau(G)\le2\nu(G)$.
\end{lemma}
An edge reduction consists of $F\ne\varnothing$, a packing $S$ and an edge set $X$ satisfying
\[
|X|\le2|S|,\ E(S)\setminus F\subseteq X,\ G-X\text{ has no triangle using }F.
\]
Puleo \cite[Definition 4.2 and Lemma 4.3]{Puleo} extends Tuza's inequality from $G-F-X$ to $G$.
\leavevmode\vspace{0.5em} \\
Assume $G$ is a counterexample with $\Delta(G)\le10$ minimizing $|V(G)|+|E(G)|$.
Every proper subgraph satisfies Tuza's inequality, so $G$ has no vertex or edge reduction.
The third condition in Definition \ref{reducing-ten} keeps every packed edge outside $U$ in $X$ and hence out of $G'$.
\leavevmode\vspace{0.5em} \\
A graph $L$ is weak K\"onig Egerv\'ary if a matching $M$ and a vertex set $Q$ satisfy
\begin{equation*}
|Q|\le |M|,\,E(L-Q)\subseteq M.
\end{equation*}
Following Puleo \cite[Definition 4.1]{Puleo}, we write WKE for this property.
\leavevmode\vspace{0.5em} \\
In a link at $x$, the edges of $M$ give $|M|$ triangles at $x$.
Selecting $M$ and the edges from $x$ to $Q$ covers every triangle at $x$ with at most $2|M|$ edges.
\begin{lemma}[Puleo {\cite[Lemmas 4.3, 4.4 and Corollary 4.9]{Puleo}}]\label{wke-reduction-ten}
Every connected graph of order at most $4$ is WKE.
A WKE link gives a vertex reduction, and a WKE component of a link gives an edge reduction.
\end{lemma}
An isolated vertex reduces with $S=X=\varnothing$.
So every link component has at least $5$ vertices and is not WKE, and $\delta(G)\ge5$.
\leavevmode\vspace{0.5em} \\
Following Puleo \cite[Definition 2.3]{Puleo}, call $G$ robust when every link component has order at least $5$.
\begin{lemma}[Puleo {\cite[Lemma 2.7]{Puleo}}]\label{lowstructure-ten}
In a robust graph with no reduction, $\{x:d(x)\le6\}$ is independent and
\[
d(x)\le6\implies
\Delta\bigl(\overline{G[N(x)]}\bigr)\le1,\,
\bigl|E\bigl(\overline{G[N(x)]}\bigr)\bigr|\ne2.
\]
If $d(x)=6$ and $G[N(x)]\cong K_{2,2,2}$, then no neighbor of $x$ has degree $7$.
\end{lemma}

\subsection{Local reductions}
For $uv\in E(G)$, put
$$
J=G[N[u]\cup N[v]]$$ $$C=N(u)\cap N(v)$$
$$A=N(u)\setminus N[v]$$ $$B=N(v)\setminus N[u].$$
Call $A,B$ the private sets.
Every triangle at $u$ or $v$ lies in $J$ and uses no edge of $E(A,B)$.
\begin{lemma}\label{interval-ten}\label{commonvariableinterval}
Let $J^-\subseteq J^+$ be graphs on $V(J)$ with the same neighborhoods at $u,v$, let $S$ be a packing in $J^-$, and let $X\subseteq E(J^+)$ satisfy
$$
|X|\le2|S|$$
$$E(S)\setminus E(\{u,v\},V(J^+))\subseteq X$$
$$J^+-X\text{ has no triangle at }u\text{ or }v.$$
If $J^-\subseteq J\subseteq J^+$, then $\{u,v\}$ reduces with $S,X\cap E(G)$.
\end{lemma}
\begin{proof}
The lower inclusion keeps every packed edge and the upper inclusion makes $X\cap E(G)$ cover every triangle at $u$ or $v$.
Every packed edge outside $\{u,v\}$ belongs to $X\cap E(G)$, with $|X\cap E(G)|\le|X|\le2|S|$.
\end{proof}
For fixed $S$ and an edge set $X$ meeting every triangle on $uv$, require $E(S)\subseteq E(J^-)$ and omit from $J^+$ every link edge completing a triangle that misses $X$.
Keep $E(A,B)$ and allow relabelings that preserve the endpoint neighborhoods.
\leavevmode\vspace{0.5em} \\
For $Z\subseteq C$, $U=\{u,v\}\cup Z$ and $\chi:V(J)\to\{0,1\}$, take
$$X_J=\{ab\in E(J):\chi(a)=\chi(b)\}$$
$$X=X_J\cup E\bigl(Z,V(G)\setminus V(J)\bigr)$$
Every local triangle meets $X_J$, and every triangle through $Z$ outside $J$ meets a selected outside edge.
So a packing $S$ in $J$ gives a reduction whenever
\[
E(S)\setminus E(U,V(G))\subseteq X_J,\,
|X_J|+\sum_{z\in Z}\bigl(10-d_J(z)\bigr)\le2|S|.
\]
Adding an edge $ab$ changes the cover bound by
\[
\mathbf1_{\chi(a)=\chi(b)}-\mathbf1_{a\in Z}-\mathbf1_{b\in Z}.
\]
This is at most zero if $\{a,b\}\cap Z\ne\varnothing$.
Give $A,B$ opposite colors so $E(A,B)$ adds zero.
Other allowed edges outside $Z$ cost one if their endpoint colors agree.
\leavevmode\vspace{0.5em} \\
An edge reduction can delete a packed edge without picking it only after covering every triangle on that edge.
\section{Reductions}
\subsection{Subcubic links}\label{subcubic-ten}
Suppose every edge of $G$ has codegree at most $3$.
Then every link $L$ has $\Delta(L)\le3$.
\leavevmode\vspace{0.5em} \\
Let $P$ be the Petersen graph.
Get $R$ by subdividing one edge in each of two copies of $K_4$ and joining the new vertices.
Write $\nu_m(L)$ for the matching number of $L$.
\begin{lemma}[{\cite{Supplement}}]\label{subcubic-basis-ten}
Let $L$ be connected with $\delta(L)\ge2$, $\Delta(L)\le3$ and $|V(L)|\le10$.
Then $L$ is WKE unless $L\cong P$ or $L\cong R$.
\end{lemma}
\begin{lemma}\label{subcubic-links-ten}\label{no-r-ten}
Every link of $G$ is isomorphic to $P$, and $d(x)=10$ for every $x\in V(G)$.
\end{lemma}
\begin{proof}
For a connected $L$ with $\Delta(L)\le3$, $|V(L)|\le10$ and leaf $a$ adjacent to $b$, induction and Lemmas \ref{wke-reduction-ten} and \ref{subcubic-basis-ten} give WKE witnesses in each component of $L-\{a,b\}$.
Combine them and add $ab$ to $M$ and $b$ to $Q$.
\leavevmode\vspace{0.5em} \\
So every connected link of order at most $9$ is WKE.
A disconnected link has two components of order $5$ and either of them gives an edge reduction.
So every link is $P$ or $R$.
\leavevmode\vspace{0.5em} \\
Assume the link at $x$ contains a diamond with centers $a,b$ and tips $c,d$.
The common neighbors of $ax$ are $b,c,d$, and those of $ab$ are $x,c,d$.
So $ax,bx,ab$ each lie in $2$ copies of $K_4$.
\leavevmode\vspace{0.5em} \\
The edge $xc$ lies in $1$ copy of $K_4$.
Its third common neighbor can't meet $a$ or $b$, since $ax,bx,ab$ have codegree $3$.
\leavevmode\vspace{0.5em} \\
In a graph with $\Delta\le3$, two triangles through a vertex share an edge.
If triangles extend two edges of a fixed triangle, their fourth vertices coincide and give $K_4$.
So the counts of triangles through the three vertices of a triangle are
\[
(1,1,1),\,(2,2,1),\,(3,3,3).
\]
In the link at $a$, the triangle $xbc$ has counts $2$ at $x,b$ and hence $1$ at $c$.
So $ac$ lies in $1$ copy of $K_4$, and applying the same argument at $b$ gives the same count for $bc$.
\leavevmode\vspace{0.5em} \\
The triangle $xab$ in the link at $c$ has counts $(1,1,1)$.
But $P$ has no triangle and every triangle of $R$ has counts $(2,2,1)$.
So no link is $R$, and every link is $P$.
\end{proof}
\begin{lemma}\label{petersen-pair-ten}
If every link of $G$ is isomorphic to $P$, then $G$ has a reducible pair.
\end{lemma}
\begin{proof}
Choose $uv\in E(G)$ with common neighbors $c_0,c_1,c_2$, and write the other neighbors as $a_0,\ldots,a_5$ at $u$ and $b_0,\ldots,b_5$ at $v$.
The common neighbors are independent.
\leavevmode\vspace{0.5em} \\
Use the neighborhood labels in the right drawing, replacing $a_i$ by $b_i$ and exchanging $c_1,c_2$ at $v$.
Represent $P$ by the two element subsets of $\{1,2,3,4,5\}$, with disjoint subsets adjacent.
The label permutations act transitively on vertices and give every permutation of the three neighbors while fixing their vertex.
\leavevmode\vspace{0.5em} \\
So these labels work for both links in any order of their common neighbors.
\leavevmode\vspace{0.5em} \\
Take $U=\{u,v\}$, the $7$ red triangles as $S$ and the $14$ blue edges as $X$.
\begin{center}
\begin{minipage}{\textwidth}
\centering\setlength{\tabcolsep}{3pt}
\begin{tabular}{cc}
\begin{tikzpicture}[x=0.53cm,y=0.61cm,
every node/.style={circle,draw,fill=white,inner sep=1.1pt,font=\scriptsize}]
\coordinate (u) at (-1.8,0);
\coordinate (v) at (1.8,0);
\coordinate (c0) at (0,2.8);
\coordinate (c1) at (0,-0.6);
\coordinate (c2) at (0,-2.1);
\coordinate (a0) at (-5.3,-2.6);
\coordinate (a1) at (-4,-1.0);
\coordinate (a2) at (-5.3,0.7);
\coordinate (a5) at (-4.3,0.94);
\coordinate (a3) at (-5.3,2);
\coordinate (a4) at (-4,2);
\coordinate (b0) at (5.3,-2.6);
\coordinate (b1) at (4,-0.6);
\coordinate (b2) at (5.3,0.7);
\coordinate (b5) at (4.3,0.94);
\coordinate (b3) at (5.3,2);
\coordinate (b4) at (4,2);
\fill[red!12] (u)--(c0)--(v)--cycle;
\fill[red!12] (v)--(c1)--(b1)--cycle;
\fill[red!12] (u)--(c2)--(a1)--cycle;
\fill[red!12] (u)--(a2)--(a5)--cycle;
\fill[red!12] (u)--(a3)--(a4)--cycle;
\fill[red!12] (v)--(b2)--(b5)--cycle;
\fill[red!12] (v)--(b3)--(b4)--cycle;
\draw[red!75!black,line width=0.7pt] (a1)--(c2);
\draw[red!75!black,line width=0.7pt] (a1)--(u);
\draw[red!75!black,line width=0.7pt] (a2)--(a5);
\draw[red!75!black,line width=0.7pt] (a2)--(u);
\draw[red!75!black,line width=0.7pt] (a3)--(a4);
\draw[red!75!black,line width=0.7pt] (a3)--(u);
\draw[red!75!black,line width=0.7pt] (a4)--(u);
\draw[red!75!black,line width=0.7pt] (a5)--(u);
\draw[red!75!black,line width=0.7pt] (b1)--(c1);
\draw[red!75!black,line width=0.7pt] (b1)--(v);
\draw[red!75!black,line width=0.7pt] (b2)--(b5);
\draw[red!75!black,line width=0.7pt] (b2)--(v);
\draw[red!75!black,line width=0.7pt] (b3)--(b4);
\draw[red!75!black,line width=0.7pt] (b3)--(v);
\draw[red!75!black,line width=0.7pt] (b4)--(v);
\draw[red!75!black,line width=0.7pt] (b5)--(v);
\draw[red!75!black,line width=0.7pt] (c0)--(u);
\draw[red!75!black,line width=0.7pt] (c0)--(v);
\draw[red!75!black,line width=0.7pt] (c1)--(v);
\draw[red!75!black,line width=0.7pt] (c2)--(u);
\draw[red!75!black,line width=0.7pt] (u)--(v);
\node at (u) {$u$};
\node at (v) {$v$};
\node at (c0) {$c_{0}$};
\node at (c1) {$c_{1}$};
\node at (c2) {$c_{2}$};
\node at (a0) {$a_{0}$};
\node at (a1) {$a_{1}$};
\node at (a2) {$a_{2}$};
\node at (a3) {$a_{3}$};
\node at (a4) {$a_{4}$};
\node at (a5) {$a_{5}$};
\node at (b0) {$b_{0}$};
\node at (b1) {$b_{1}$};
\node at (b2) {$b_{2}$};
\node at (b3) {$b_{3}$};
\node at (b4) {$b_{4}$};
\node at (b5) {$b_{5}$};
\end{tikzpicture}&
\begin{tikzpicture}[x=0.53cm,y=0.61cm,
every node/.style={circle,draw,fill=white,inner sep=1.1pt,font=\scriptsize}]
\coordinate (u) at (-1.8,0);
\coordinate (v) at (1.8,0);
\coordinate (c0) at (0,2.8);
\coordinate (c1) at (0,-0.6);
\coordinate (c2) at (0,-2.1);
\coordinate (a0) at (-5.3,-2.6);
\coordinate (a1) at (-4,-1.0);
\coordinate (a2) at (-5.3,0.7);
\coordinate (a5) at (-4.3,0.94);
\coordinate (a3) at (-5.3,2);
\coordinate (a4) at (-4,2);
\coordinate (b0) at (5.3,-2.6);
\coordinate (b1) at (4,-0.6);
\coordinate (b2) at (5.3,0.7);
\coordinate (b5) at (4.3,0.94);
\coordinate (b3) at (5.3,2);
\coordinate (b4) at (4,2);
\draw[black!28,line width=0.25pt] (a0) .. controls (-6.4,-2.6) and (-6.4,2) .. (a3);
\draw[black!28,line width=0.25pt] (a0)--(a5);
\draw[black!28,line width=0.25pt] (a0)--(c2);
\draw[black!28,line width=0.25pt] (a0)--(u);
\draw[black!28,line width=0.25pt] (a1)--(a2);
\draw[black!28,line width=0.25pt] (a1) .. controls (-3.15,-1) and (-3.15,2) .. (a4);
\draw[black!28,line width=0.25pt] (a1)--(c2);
\draw[black!28,line width=0.25pt] (a1)--(u);
\draw[black!28,line width=0.25pt] (a2)--(a5);
\draw[black!28,line width=0.25pt] (a2) .. controls (-4,-1.2) and (0,-1.2) .. (c1);
\draw[black!28,line width=0.25pt] (a2)--(u);
\draw[black!28,line width=0.25pt] (a3)--(a4);
\draw[black!28,line width=0.25pt] (a3) .. controls (-4.2,3.5) and (-1.5,2.8) .. (c1);
\draw[black!28,line width=0.25pt] (a3)--(u);
\draw[black!28,line width=0.25pt] (a4)--(c0);
\draw[black!28,line width=0.25pt] (a4)--(u);
\draw[black!28,line width=0.25pt] (a5)--(c0);
\draw[black!28,line width=0.25pt] (a5)--(u);
\draw[black!28,line width=0.25pt] (b0) .. controls (6.4,-2.6) and (6.4,2) .. (b3);
\draw[black!28,line width=0.25pt] (b0)--(b5);
\draw[black!28,line width=0.25pt] (b0)--(c1);
\draw[black!28,line width=0.25pt] (b0)--(v);
\draw[black!28,line width=0.25pt] (b1)--(b2);
\draw[black!28,line width=0.25pt] (b1) .. controls (3.15,-0.6) and (3.15,2) .. (b4);
\draw[black!28,line width=0.25pt] (b1)--(c1);
\draw[black!28,line width=0.25pt] (b1)--(v);
\draw[black!28,line width=0.25pt] (b2)--(b5);
\draw[black!28,line width=0.25pt] (b2)--(c2);
\draw[black!28,line width=0.25pt] (b2)--(v);
\draw[black!28,line width=0.25pt] (b3)--(b4);
\draw[black!28,line width=0.25pt] (b3) .. controls (5.1,-0.2) and (1,-1.8) .. (c2);
\draw[black!28,line width=0.25pt] (b3)--(v);
\draw[black!28,line width=0.25pt] (b4)--(c0);
\draw[black!28,line width=0.25pt] (b4)--(v);
\draw[black!28,line width=0.25pt] (b5)--(c0);
\draw[black!28,line width=0.25pt] (b5)--(v);
\draw[black!28,line width=0.25pt] (c0)--(u);
\draw[black!28,line width=0.25pt] (c0)--(v);
\draw[black!28,line width=0.25pt] (c1)--(u);
\draw[black!28,line width=0.25pt] (c1)--(v);
\draw[black!28,line width=0.25pt] (c2)--(u);
\draw[black!28,line width=0.25pt] (c2)--(v);
\draw[black!28,line width=0.25pt] (u)--(v);
\draw[blue!75!black,line width=0.8pt] (a0)--(u);
\draw[blue!75!black,line width=0.8pt] (a1)--(c2);
\draw[blue!75!black,line width=0.8pt] (a1)--(u);
\draw[blue!75!black,line width=0.8pt] (a2)--(a5);
\draw[blue!75!black,line width=0.8pt] (a3)--(a4);
\draw[blue!75!black,line width=0.8pt] (b0)--(v);
\draw[blue!75!black,line width=0.8pt] (b1)--(c1);
\draw[blue!75!black,line width=0.8pt] (b1)--(v);
\draw[blue!75!black,line width=0.8pt] (b2)--(b5);
\draw[blue!75!black,line width=0.8pt] (b3)--(b4);
\draw[blue!75!black,line width=0.8pt] (c0)--(u);
\draw[blue!75!black,line width=0.8pt] (c0)--(v);
\draw[blue!75!black,line width=0.8pt] (c1)--(u);
\draw[blue!75!black,line width=0.8pt] (c2)--(v);
\node at (u) {$u$};
\node at (v) {$v$};
\node at (c0) {$c_{0}$};
\node at (c1) {$c_{1}$};
\node at (c2) {$c_{2}$};
\node at (a0) {$a_{0}$};
\node at (a1) {$a_{1}$};
\node at (a2) {$a_{2}$};
\node at (a3) {$a_{3}$};
\node at (a4) {$a_{4}$};
\node at (a5) {$a_{5}$};
\node at (b0) {$b_{0}$};
\node at (b1) {$b_{1}$};
\node at (b2) {$b_{2}$};
\node at (b3) {$b_{3}$};
\node at (b4) {$b_{4}$};
\node at (b5) {$b_{5}$};
\end{tikzpicture}\\[0.3em]
The packing $S$&The cover $X$
\end{tabular}

\end{minipage}
\end{center}

The $7$ red triangles don't share any edges, and $X$ contains every packed edge outside $U$.
The blue edges incident with $c_0,c_1,c_2$ cover every triangle on $uv$.
At $u$, the edges $ua_0,ua_1$ cover every triangle using $a_0$ or $a_1$.
\leavevmode\vspace{0.5em} \\
Of the other edges between the $a_i$, only $a_2a_5,a_3a_4$ remain, and both lie in $X$.
The only common neighbor $c_i$ with $uc_i\notin X$ is $c_2$, and its other neighbors in the link are $a_0,a_1$.
The same argument works at $v$ after exchanging $c_1,c_2$.
\leavevmode\vspace{0.5em} \\
Edges between the two private sets lie in no triangle at $u$ or $v$.
So $X$ covers every triangle meeting $U$, with $|X|=14=2|S|$, and $U$ is reducible.
\end{proof}
So $G$ has an edge of codegree at least $4$.

\subsection{Vertices of small degree}
\begin{lemma}\label{five-component-ten}
If every edge of $G$ has codegree at most $4$, then no link component has order $5$.
\end{lemma}
\begin{proof}
Let $L$ be connected with $|V(L)|=5$ and not WKE.
A star is WKE, and otherwise $\nu_m(L)=2$.
An independent triple gives a cover of size $2$ and a WKE witness.
\leavevmode\vspace{0.5em} \\
Say a triple spans only $ab$, with third vertex $c$ and complement $Q$.
An edge in $L[Q\cup\{c\}]$ forms a matching with $ab$, and $Q$ covers every other edge.
If no such edge exists, $Q\cup\{c\}$ is independent.
\leavevmode\vspace{0.5em} \\
So every triple spans at least $2$ edges and $\Delta(\overline L)\le1$.
Thus $L\in\{K_5,K_5-e,W_4\}$, and none is WKE.
\leavevmode\vspace{0.5em} \\
Each edge $ab$ in a row headed $h$ indicates the triangle $hab$.
\leavevmode\vspace{0.5em} \\
Let $L$ be a component of the link at $x$.
\leavevmode\vspace{0.5em} \\
\noindent\textbf{Case 1: $L=K_5$.}\vspace{0.5em} \\*
Adding $x$ gives a $K_6$.
Each edge has $4$ common neighbors there, so no triangle using its edges has a vertex outside the $K_6$.
\leavevmode\vspace{0.5em} \\
Label the $K_6$ by $1,\ldots,6$.
Delete $E(K_6)$ and take the below packing with
\[
X=E(G[\{1,2,3\}])\cup E(G[\{4,5,6\}]),\,|X|=6\le8=2|S|.
\]
The graph $K_6-X$ is bipartite, so this gives an edge reduction.
\begin{center}
\begin{minipage}{\textwidth}
\centering\setlength{\tabcolsep}{3pt}
\begin{tabular}{@{}c@{}}
\begin{tabular}{@{}l@{}}
\PackingRow{1}{2/3,4/5}\\[0.18em]
\PackingRow{2}{4/6}\\[0.18em]
\PackingRow{3}{5/6}\\[0.18em]
\end{tabular}
\end{tabular}

\end{minipage}
\end{center}
\leavevmode\vspace{0.5em} \\
\noindent\textbf{Case 2: $L=K_5-ab$.}\vspace{0.5em} \\*
Write its other vertices as $u,c,d$.
Take the below packing and
\[
X=E(\{a,b\},\{u,c,d\})\cup\{cd,xu\},\,|X|=8=2|S|.
\]
\begin{center}
\begin{minipage}{\textwidth}
\centering\setlength{\tabcolsep}{3pt}
\begin{tabular}{@{}c@{}}
\begin{tabular}{@{}l@{}}
\PackingRow{a}{x/c,u/d}\\[0.18em]
\PackingRow{b}{x/d,u/c}\\[0.18em]
\end{tabular}
\end{tabular}

\end{minipage}
\end{center}
\leavevmode\vspace{0.5em} \\
Delete all edges of this $K_6-ab$.
A triangle with a vertex outside the $K_6-ab$ can't use an edge within $\{x,u,c,d\}$ by the codegree bound, or $xa,xb$ by the link component condition.
The set $X$ covers every other triangle using a deleted edge.
\leavevmode\vspace{0.5em} \\
\noindent\textbf{Case 3: $L=W_4$.}\vspace{0.5em} \\*
Write $w$ for the center and $1,2,3,4$ for the cycle vertices in order.
Delete all edges on $V(L)\cup\{x\}$ and take the below packing with $X=E(L)$.
We have $|X|=8=2|S|$.
\begin{center}
\begin{minipage}{\textwidth}
\centering\setlength{\tabcolsep}{3pt}
\begin{tabular}{@{}c@{}}
\begin{tabular}{@{}l@{}}
\PackingRow{x}{1/2,3/4}\\[0.18em]
\PackingRow{w}{1/4,2/3}\\[0.18em]
\end{tabular}
\end{tabular}

\end{minipage}
\end{center}
\leavevmode\vspace{0.5em} \\
No triangle on $xw$ has a vertex outside this graph by the codegree bound, and none on an edge at $x$ does by the link component condition.
The set $X$ covers every other triangle using a deleted edge.
\end{proof}
Let $G[N(x)]\cong K_{2,2,2}$, take $u\in N(x)$, and write $w$ for the vertex of $N(x)$ not adjacent to $u$.
Label
\[
x=0,\,u=1,\,w=2,\,C=\{3,4,5,6\},\,C\text{ has parts }\{3,4\},\{5,6\}.
\]
Write $B=N(u)\setminus\bigl(C\cup\{x\}\bigr)$.
Packings A, B and C pair the below triangles $S$ with the blue cover $X$.
Each $X$ contains every packed edge outside $\{x,u\}$, and relabeling the cycle on $C$ lets Packing B leave any prescribed $uq$ unused.
\begin{center}
\begin{minipage}{\textwidth}
\centering
\setlength{\tabcolsep}{3pt}
\begin{tabular}{@{}ccc@{}}
\begin{minipage}[t]{0.30\textwidth}\centering\textbf{A} $|S|=4,\ |X|=6$\\[0.3em]
\begin{tabular}{@{}l@{}}
\PackingRow{0}{3/6,4/5}\\[0.18em]
\PackingRow{1}{3/5,4/6}\\[0.18em]
\end{tabular}\\[0.3em]
\begin{tikzpicture}[baseline=-0.5ex]
\foreach \GraphVertex/\GraphX/\GraphY in {0/0.00000/1.00000,1/0.78183/0.62349,2/0.97493/-0.22252,3/0.43388/-0.90097,4/-0.43388/-0.90097,5/-0.97493/-0.22252,6/-0.78183/0.62349}{\coordinate (p\GraphVertex) at (\GraphX,\GraphY);}
\foreach \GraphFrom/\GraphTo in {0/1,0/2,3/5,3/6,4/5,4/6}{\draw[blue!75!black,line width=0.7pt] (p\GraphFrom)--(p\GraphTo);}
\foreach \GraphVertex/\GraphLabel in {0/{0},1/{1},2/{2},3/{3},4/{4},5/{5},6/{6}}{\node[proof vertex] at (p\GraphVertex) {$\GraphLabel$};}
\end{tikzpicture}\end{minipage}&
\begin{minipage}[t]{0.30\textwidth}\centering\textbf{B} $|S|=4,\ |X|=9$\\[0.3em]
\begin{tabular}{@{}l@{}}
\PackingRow{0}{1/6,2/4,3/5}\\[0.18em]
\PackingRow{1}{4/5}\\[0.18em]
\end{tabular}\\[0.3em]
\begin{tikzpicture}[baseline=-0.5ex]
\foreach \GraphVertex/\GraphX/\GraphY in {0/0.00000/1.00000,1/0.78183/0.62349,2/0.97493/-0.22252,3/0.43388/-0.90097,4/-0.43388/-0.90097,5/-0.97493/-0.22252,6/-0.78183/0.62349}{\coordinate (p\GraphVertex) at (\GraphX,\GraphY);}
\foreach \GraphFrom/\GraphTo in {0/2,0/6,1/3,1/4,1/5,1/6,2/4,3/5,4/5}{\draw[blue!75!black,line width=0.7pt] (p\GraphFrom)--(p\GraphTo);}
\foreach \GraphVertex/\GraphLabel in {0/{0},1/{1},2/{2},3/{3},4/{4},5/{5},6/{6}}{\node[proof vertex] at (p\GraphVertex) {$\GraphLabel$};}
\end{tikzpicture}\end{minipage}&
\begin{minipage}[t]{0.30\textwidth}\centering\textbf{C} $|S|=5,\ |X|=9$\\[0.3em]
\begin{tabular}{@{}l@{}}
\PackingRow{0}{1/3,2/5,4/6}\\[0.18em]
\PackingRow{1}{4/5}\\[0.18em]
\PackingRow{2}{3/6}\\[0.18em]
\end{tabular}\\[0.3em]
\begin{tikzpicture}[baseline=-0.5ex]
\foreach \GraphVertex/\GraphX/\GraphY in {0/0.00000/1.00000,1/0.78183/0.62349,2/0.97493/-0.22252,3/0.43388/-0.90097,4/-0.43388/-0.90097,5/-0.97493/-0.22252,6/-0.78183/0.62349}{\coordinate (p\GraphVertex) at (\GraphX,\GraphY);}
\foreach \GraphFrom/\GraphTo in {0/1,2/3,2/4,2/5,2/6,3/5,3/6,4/5,4/6}{\draw[blue!75!black,line width=0.7pt] (p\GraphFrom)--(p\GraphTo);}
\foreach \GraphVertex/\GraphLabel in {0/{0},1/{1},2/{2},3/{3},4/{4},5/{5},6/{6}}{\node[proof vertex] at (p\GraphVertex) {$\GraphLabel$};}
\end{tikzpicture}\end{minipage}
\end{tabular}

\end{minipage}
\end{center}

\begin{lemma}\label{thin-eight-nine-ten}
If $d(u)\in\{8,9\}$ and $G[N(u)]$ is connected, then $\{x,u\}$ is reducible.
\end{lemma}
\begin{proof}
\leavevmode\vspace{0.5em} \\
\noindent\textbf{Case 1: $\nu_m(G[B])\ge2$.}\vspace{0.5em} \\*
Take a matching $M$ of size $2$ in $G[B]$ and add its $2$ triangles at $u$ to Packing A.
Pick $M\cup E(\{u\},B)$ to get $|S|=6$ and $|X|\le12$, since $|B|\in\{3,4\}$.
\leavevmode\vspace{0.5em} \\
Assume $\nu_m(G[B])\le1$.
\leavevmode\vspace{0.5em} \\
\noindent\textbf{Case 2: $E(G[B])=\varnothing$.}\vspace{0.5em} \\*
Connectivity gives $bq\in E(B,C)$.
Take Packing B with $uq$ not used and add $ubq$ and the cover edge $bq$.
We get $|S|=5$, $|X|=10$.
\leavevmode\vspace{0.5em} \\
\noindent\textbf{Case 3: $G[B]$ is a star with isolates.}\vspace{0.5em} \\*
If $|E(G[B])|=1$, add its triangle at $u$ to Packing B and select the edge.
If $G[B]$ has a star centered at $a$ and no edge joins $B\setminus\{a\}$ to $C$, add $ua$ to the cover of Packing C.
Both choices give $|S|=5$ and $|X|=10$.
\leavevmode\vspace{0.5em} \\
Otherwise take $bq\in E(B\setminus\{a\},C)$ and a star edge $ac$ not incident with $b$.
Add $uac,ubq$ to Packing B with $uq$ not used, and pick $ua,ac,bq$.
We get $|S|=6$, $|X|=12$.
\leavevmode\vspace{0.5em} \\
\noindent\textbf{Case 4: $G[B]$ contains a triangle.}\vspace{0.5em} \\*
If $|B|=4$, write $a,b,c$ for the triangle vertices and $d$ for the isolate.
Connectivity gives $dq\in E(B,C)$.
Add $uab,udq$ to Packing B with $uq$ not used and pick $uc,ab,dq$.
\leavevmode\vspace{0.5em} \\
If $|B|=3$, take $dq\in E(B,C)$ and write $a,b$ for the other vertices of $B$.
Add $uab,udq$ to Packing B with $uq$ not used and pick $ud,ab,dq$.
Both choices give $|S|=6$ and $|X|=12$.
\end{proof}
\begin{lemma}\label{thin-ten-ten}
Assume $d(u)=10$, $G[N(u)]$ is connected and $|N(c)\cap B|\le1$ for every $c\in C$.
Then a reduction removes $x$ or $u$.
\end{lemma}
\begin{proof}
We have $|B|=5$, $|E(B,C)|\le4$, and every component of $G[B]$ meets $C$.
\leavevmode\vspace{0.5em} \\
\noindent\textbf{Case 1: $\nu_m(G[B])\le1$.}\vspace{0.5em} \\*
Use the star constructions of Lemma \ref{thin-eight-nine-ten}.
If $G[B]$ contains a triangle, take either isolate $d$ and an edge $dq\in E(B,C)$ and use that lemma's triangle construction.
\leavevmode\vspace{0.5em} \\
\noindent\textbf{Case 2: $\nu_m(G[B])=2$.}\vspace{0.5em} \\*
Say a matching $M$ of size $2$ leaves $b$ unmatched and $N(b)\cap C=\varnothing$.
Add the $2$ triangles at $u$ on $M$ to Packing A and select $M\cup\bigl(E(\{u\},B)\setminus\{ub\}\bigr)$.
We get $|S|=6$, $|X|=12$.
\leavevmode\vspace{0.5em} \\
Let $J$ have vertex set $B\cup C$ and edge set $E(G[B])\cup E(B,C)$.
If $J$ has a cover $Z$ with $|Z|\le3$, add the $2$ triangles at $u$ on a matching in $G[B]$ to Packing C and pick the matching edges and $uZ$.
We get $|S|=7$, $|X|\le14$.
\leavevmode\vspace{0.5em} \\
So assume every vertex left unmatched by a maximum matching has a neighbor in $C$, and $J$ has no cover of size $3$.
Take a maximum matching $M$ leaving $b$ unmatched, an edge $bc\in E(B,C)$, a cycle neighbor $q$ of $c$, and the other cycle vertices $p,r$.
If $T\subseteq B$ has size $2$, covers $G[B]-M$ and contains the neighbor of $q$ in $B$ when one exists, then
\[
M'=M\cup\{bc,xq,pr\},\,Q=(C\setminus\{q\})\cup T
\]
gives a WKE witness of size $5$.
\leavevmode\vspace{0.5em} \\
Since $|N(q)\cap B|\le1$, every edge of the link outside $Q$ lies in $M'$.
\leavevmode\vspace{0.5em} \\
Apply the Gallai Edmonds theorem \cite{Kotlov} to $G[B]$.
Let $D$ be the vertices left unmatched by some maximum matching, $A=N_B(D)\setminus D$ and $R=B\setminus(D\cup A)$.
The components of $G[D]$ are factor critical, $G[R]$ has a perfect matching, and $G[D]$ has $|A|+1$ components.
\leavevmode\vspace{0.5em} \\
Every vertex of $D$ has a neighbor in $C$, so $|D|\le4$.
\leavevmode\vspace{0.5em} \\
\noindent\textbf{Subcase 2a: $A=\varnothing$, $D=\{b\}$.}\vspace{0.5em} \\*
The set $N(b)\cap C$ is nonempty and proper, so take a cycle edge $cq$ with $c\in N(b)$ and $q\notin N(b)$.
Take a perfect matching $M$ of $G[R]$ and let $T$ be the endpoints of its edge containing the neighbor of $q$ in $B$, or of either edge if $q$ has no such neighbor.
Then $T$ covers $G[B]-M$.
\leavevmode\vspace{0.5em} \\
\noindent\textbf{Subcase 2b: $A=\varnothing$, $G[D]=K_3$.}\vspace{0.5em} \\*
The set $R$ spans an edge, and the $4$ edges of $E(B,C)$ meet each vertex of $D$ once and $R$ once.
Leave a triangle vertex $b$ unmatched and let $M$ consist of the opposite triangle edge and the edge of $G[R]$.
Choose $q$ with no neighbor in $R$ and let $T=D\setminus\{b\}$.
\leavevmode\vspace{0.5em} \\
\noindent\textbf{Subcase 2c: $A=\{a\}$, $D=\{b,d\}$.}\vspace{0.5em} \\*
Write $R=\{f,g\}$, take a cycle edge $cq$ with $c\in N(b)$ and $q\notin N(b)$, and put $M=\{ad,fg\}$.
If $q$ has neighbor $z\in\{d,f,g\}$, take $T=\{a,z\}$.
Otherwise take $T=\{a,f\}$.
\leavevmode\vspace{0.5em} \\
\noindent\textbf{Subcase 2d: $A=\{a\}$, $G[D]\cong K_3\cup K_1$.}\vspace{0.5em} \\*
Write $d$ for the isolate.
Each vertex of $D$ has $1$ neighbor in $C$.
Each component of $G[D]$ contains a vertex left unmatched by some maximum matching, so $ad$ exists and $a$ has a neighbor in the triangle.
\leavevmode\vspace{0.5em} \\
If $ab\notin E(G)$ for a triangle vertex $b$, take $M$ as $ad$ and the opposite triangle edge, let $T$ be the other two triangle vertices, and choose $q\notin N(d)$.
Otherwise leave $d$ unmatched, choose either cycle neighbor $q$ of its neighbor in $C$, and write $t$ for the triangle vertex adjacent to $q$.
Take $M$ as $at$ and the opposite triangle edge and $T=\{a,t\}$.
\leavevmode\vspace{0.5em} \\
\noindent\textbf{Subcase 2e: $|A|=2$.}\vspace{0.5em} \\*
The graph $G[D]$ has $3$ isolated vertices.
If $N(D)\cap C\ne C$, take a cycle edge $cq$ with $c\in N(D)$ and $q\notin N(D)$, leave a vertex of $D$ adjacent to $c$ unmatched, and take $T=A$.
Assume now $N(D)\cap C=C$.
\leavevmode\vspace{0.5em} \\
Label one cycle part $c_1,c_2$ with distinct neighbors $d_1,d_2\in D$, and take $c_3d_3\in E(B,C)$ on the other part.
Write $c_4$ for its remaining vertex and $A=\{a_1,a_2\}$.
Since $J$ has no cover of size $3$, $a_1a_2\in E(G)$.
\[
M'=\{c_1d_1,c_2d_2,c_3d_3,a_1a_2,xc_4\},\,
Q=\{x,c_3,c_4,a_1,a_2\}
\]
is a WKE witness.
So the link at $u$ is WKE.
\end{proof}
\begin{lemma}\label{c4-low}
If the maximum codegree is $4$, then $\delta(G)\ge7$ and every link is connected.
\end{lemma}
\begin{proof}
Lemma \ref{five-component-ten} excludes link components of order $5$, so robustness excludes vertices of degree $5$ and disconnected links.
By Lemma \ref{lowstructure-ten}, every link of order $6$ is $K_{2,2,2}$ under the codegree bound.
\leavevmode\vspace{0.5em} \\
Its neighbors have degree at least $8$ by Lemma \ref{lowstructure-ten}, and Lemma \ref{thin-eight-nine-ten} excludes degrees $8,9$.
For degree $10$, each $c\in C$ has the $3$ link neighbors $x$ and its two cycle neighbors, so $|N(c)\cap B|\le1$ and Lemma \ref{thin-ten-ten} gives a reduction.
\end{proof}
\begin{lemma}\label{c5-six}
If the maximum codegree is $5$, then no edge of codegree $5$ has an endpoint of degree $6$.
\end{lemma}
\begin{proof}
Say $x=0$ has degree $6$ and $xu$ has codegree $5$, with $u=1$.
Put $C=\{2,3,4,5,6\}$ and $B=N(u)\setminus\bigl(C\cup\{x\}\bigr)$.
By Lemma \ref{lowstructure-ten}, $G[C]$ is $K_5$ or $K_5-23$, and $|B|\le4$.
\leavevmode\vspace{0.5em} \\
\noindent\textbf{Case 1: $G[C]=K_5$.}\vspace{0.5em} \\*
The codegree bound gives $E(B,C)=\varnothing$, and robustness gives $B=\varnothing$.
Take the below packing of size $7$ and $X=E(G[C])\cup\{xu\}$, with $|X|=11$.
\begin{center}
\begin{minipage}{\textwidth}
\centering\setlength{\tabcolsep}{3pt}
\begin{tabular}{@{}ccccccc@{}}
\begin{tabular}{@{}l@{}}
\PackingRow{0}{2/4,3/5,1/6}\\[0.18em]
\PackingRow{3}{4/6}\\[0.18em]
\PackingRow{1}{4/5,2/3}\\[0.18em]
\PackingRow{2}{5/6}\\[0.18em]
\end{tabular}
\end{tabular}

\end{minipage}
\end{center}
\leavevmode\vspace{0.5em} \\
\noindent\textbf{Case 2: $G[C]=K_5-23$.}\vspace{0.5em} \\*
Get rid of the triangle $123$ and take $X_0=E(G[C])\cup\{xu\}$, with $|X_0|=10$.
Only $2,3$ can have neighbors in $B$, at most $1$ each, and $u2,u3$ aren't used.
A matching $M$ of size $2$ in $G[B]$ gives $2$ more triangles at $u$, and selecting $M\cup E(\{u\},B)$ gives $|S|=8$ and $|X|\le16$.
\leavevmode\vspace{0.5em} \\
If $E(G[B])=\varnothing$, add $u2,u3$ to $X_0$.
Otherwise assume $\nu_m(G[B])=1$, so its edges form a star or a triangle.
Add isolated vertices with no neighbors in $C$ until $|B|=4$.
\leavevmode\vspace{0.5em} \\
For a star, take a leaf or isolate $b_0$ with $N(b_0)\cap C=\varnothing$ and select $E(\{u\},B)\setminus\{ub_0\}$.
Add the triangle at $u$ on the edge of $G[B]$ incident with $b_0$, or on any edge if $b_0$ is isolated, and pick that edge.
We get $|S|=7$, $|X|\le14$.
\leavevmode\vspace{0.5em} \\
For a triangle and isolate $d$ with $N(d)\cap C=\varnothing$, select $E(\{u\},B)\setminus\{ud\}$ and add the triangle at $u$ on any edge of $G[B]$, picking that edge.
\leavevmode\vspace{0.5em} \\
If $dq\in E(B,C)$, add a triangle at $u$ on an edge of $G[B]$ and the triangle $udq$, then select both new link edges and $E(\{u\},B)$.
We get $|S|=8$, $|X|=16$.
No packed edge uses an added isolate, and removing its edge at $u$ from the cover gives a reduction in $G$.
\end{proof}
\subsection{Maximum codegree}\label{choose-ten}
Choose $uv\in E(G)$ of maximum codegree $c$ with $d(u)\le d(v)$ and put $H=G[C]$.
We have $4\le c\le9$, $|C|=c$, $|A|=d(u)-c-1$ and $|B|=d(v)-c-1$.
Call the edges of $E(A,C)\cup E(B,C)$ contacts.
\leavevmode\vspace{0.5em} \\
Every triangle at $u$ or $v$ uses $uv$, an edge of $H$, a contact or a private edge.
When a case needs maximal $|E(H)|$, choose $uv$ to maximize $|E(H)|$ among edges of codegree $c$.
For $c=4$ and the selected star cases, break ties by minimizing
\[
\bigl(|E(G[A])|+|E(G[B])|,\,\min\{|E(G[A])|,|E(G[B])|\}\bigr)
\]
lexicographically.
\begin{lemma}\label{endpoint-degrees-ten}
The links at $u,v$ are connected, and
\[
d(u),d(v)\in\begin{cases}
\{7,\ldots,10\}&4\le c\le6,\\
\{c+1,\ldots,10\}&7\le c\le9.
\end{cases}
\]
\end{lemma}
\begin{proof}
For $c=4$, Lemma \ref{c4-low} gives $\delta(G)\ge7$ and connected links.
For $c=5$, Lemma \ref{c5-six} excludes an endpoint of degree $6$.
For $c\ge6$, $d(u),d(v)\ge c+1$.
\leavevmode\vspace{0.5em} \\
When $c\ge5$, the vertices $\{v\}\cup C$ lie in a component of the link at $u$ of order at least $6$.
A second component has order at least $5$, contradicting $d(u)\le10$.
The same argument applies at $v$.
\end{proof}
For $c\ge5$, the degree bound concerns only $u,v$, and other vertices may have degree $5$ or $6$.
The local graph has
\[
|N[u]\cup N[v]|=d(u)+d(v)-c\le20-c\le16.
\]
The unordered endpoint degree pairs in Lemma \ref{endpoint-degrees-ten} number
\[
3\binom{5}{2}+\binom{4}{2}+\binom{3}{2}+1=40.
\]
\section{Edge neighborhoods}\label{c6-human-ten}
A matching in $H$ gives triangles at either endpoint.
Two edge disjoint matchings let us pack at both endpoints but we also need to cover triangles on $uv$ and contacts to the private sets.
\begin{lemma}\label{one-hub-ten}
Let $L$ be connected with $V(L)=\{h\}\dot\cup C\dot\cup B$, $E(L[C])=\varnothing$, $|C|\ge5$, $|B|\le3$ and $N_L(h)=C$.
Then $L$ is not WKE if and only if $|B|=3$, $L[B]\in\{P_3,K_3\}$ and all three vertices of $B$ have the same two neighbors in $C$.
\end{lemma}
\begin{proof}
Let $k$ be the matching number between $B,C$.
Each contact matching leaves a vertex $r\in C$ for the edge $hr$.
\leavevmode\vspace{0.5em} \\
\noindent\textbf{Case 1: $k=0$ or $k=3$.}\vspace{0.5em} \\*
For $k=3$, take three contact edges and $hr$ with $Q=\{h\}\cup B$.
For $k=0$, connectivity gives $B=\varnothing$, so use $hr$ and $Q=\{h\}$.
\leavevmode\vspace{0.5em} \\
\noindent\textbf{Case 2: $k=1$.}\vspace{0.5em} \\*
The contacts form a star.
If its center is $b\in B$, take $Q=\{h,b\}$ and match $hr$ with the edge on $B\setminus\{b\}$ when present, or with a contact at $b$.
If its center is $c\in C$ and $|E(L[B])|\le1$, use $Q=\{h,c\}$ and $hr$ with that private edge or any contact.
\leavevmode\vspace{0.5em} \\
For a path $abd$ with $ac$ present, take $M=\{ac,bd,hr\}$ and $Q=\{h,c,b\}$.
If only $b$ contacts $C$, use the preceding star witness.
For a triangle $abd$, take $M=\{bc,ad,hr\}$ and the same $Q$.
\leavevmode\vspace{0.5em} \\
\noindent\textbf{Case 3: $k=2$.}\vspace{0.5em} \\*
K\"onig's theorem \cite{LovaszPlummer} gives a contact cover of size two.
Use $B$ when $|B|=2$.
A cover contained in $B$, together with $h$, works with a maximum contact matching and $hr$.
\leavevmode\vspace{0.5em} \\
For a cover $\{b,c\}$ with $b\in B,c\in C$, write $B=\{b,d,e\}$ and take $Q=\{h,b,c\}$.
If $de$ is absent, use the contact matching and $hr$.
Otherwise $k=2$ gives $bq$ with $q\ne c$, and $M=\{de,bq,hr\}$ works.
\leavevmode\vspace{0.5em} \\
For a cover $\{c,d\}\subset C$, use $Q=\{h,c,d\}$.
If $E(L[B])=\varnothing$, use two contact edges and $hr$.
If $E(L[B])=\{ab\}$, the third vertex $e$ has a contact, so take $M=\{ab,ec,hr\}$ after exchanging $c,d$.
\leavevmode\vspace{0.5em} \\
For a path $abe$ with $ad$ absent, take $Q=\{h,c,b\}$ and $M=\{ed,ba,hr\}$ when $ed$ exists, or a maximum contact matching and $hr$ otherwise.
If $a,e$ have all four contacts but $bd$ is absent, take $Q=\{h,a,e\}$ and $M=\{bc,ad,hr\}$ when $bc$ exists, or $M=\{ba,ec,hr\}$ otherwise.
\leavevmode\vspace{0.5em} \\
For a triangle $abe$ with $ad$ absent, take $Q=\{h,b,e\}$ and $M=\{ac,be,hr\}$ when $ac$ exists.
If $a$ has no contact, relabel $b,e$ and take $M=\{ab,ec,hr\}$.
\leavevmode\vspace{0.5em} \\
The only graphs left have $K_{3,2}$ between $B,C$ and $L[B]\in\{P_3,K_3\}$.
Here $\nu_m(L)=3$, and any WKE set contains $h$ since $|C|\ge5$.
Deleting at most two more vertices leaves two incident edges on $B$ and its two contacted vertices in $C$, so these graphs are not WKE.
\end{proof}
\begin{lemma}\label{two-hub-ten}
In Lemma \ref{one-hub-ten} drop the connected condition, let $|C|=5$ and add $k$ with $hk\in E$, $|N(k)\cap C|\in\{4,5\}$ and $N(k)\cap B=\varnothing$ when $|N(k)\cap C|=5$.
The resulting connected graph is WKE unless $L[B]\in\{P_3,K_3\}$ and all three vertices of $B$ have the same two neighbors in $C$.
\end{lemma}
\begin{proof}
If $L-k$ is connected, add $k$ to $Q$ and match $k$ to an unused adjacent leaf.
Only a witness with three contacts uses four leaves, and its edge at $h$ can use the leaf missing $k$ when free.
Every other witness leaves at least two leaves not used.
\leavevmode\vspace{0.5em} \\
If $L-k$ is disconnected, the component of $h$ has at most two vertices of $B$, so its witness uses at most three leaves.
The other components lie in $B$ and are WKE.
Combine their witnesses, add $k$ to $Q$ and match $k$ to an unused adjacent leaf.
\end{proof}
A picked edge of codegree $6$ can't have an independent common neighborhood.
Lemma \ref{one-hub-ten} forces both private groups to have three vertices with each group sharing two neighbors in $C$.
So some vertex of $C$ misses both groups and has a component of order two in its link which contradicts robustness.
\begin{lemma}\label{three-hub-five-ten}
Let $u,v,w$ form a triangle and join the same $5$ independent vertices $C$, with three further neighbors at each of $u,v,w$ forming disjoint groups, each containing an edge and having complete contacts to a pair of vertices in $C$.
Then $\{u,v,w\}$ is reducible.
\end{lemma}
\begin{proof}
For $h\in\{u,v,w\}$, name its private vertices $a_h,b_h,c_h$ with $a_hb_h$ present and choose one contacted leaf $p_h$.
Pack $ha_hb_h,hc_hp_h$ and pick $ha_h,hb_h,hc_h,a_hb_h,c_hp_h$.
For each pair $h,k$ from $\{u,v,w\}$, at least three leaves avoid $p_h,p_k$.
\leavevmode\vspace{0.5em} \\
Pack one triangle on each of $uv,uw,vw$ using distinct leaves, and pick these three edges.
We get $|S|=9$ and $|X|=18$. Every triangle meeting $\{u,v,w\}$ meets $X$, and every packed edge outside these vertices belongs to $X$.
\end{proof}
The common graph $K_{1,5}$ also gives a reduction.
Let $w$ be its center and $C$ its $5$ leaves.
The codegree bound forbids a private contact to $w$ in either endpoint link.
\leavevmode\vspace{0.5em} \\
Lemma \ref{two-hub-ten} gives three private vertices at each endpoint, inducing $P_3$ or $K_3$ and sharing two neighbors in $C$.
The link at $w$ has the component of order $7$ on $u,v,C$, so by robustness it's connected.
\leavevmode\vspace{0.5em} \\
Its further neighbors miss $u,v$, and Lemma \ref{two-hub-ten} gives the third private group.
The endpoint neighborhoods make these groups disjoint, so apply Lemma \ref{three-hub-five-ten}.
\begin{lemma}\label{four-private-ten}
Let $L$ be connected with $V(L)=\{h,k\}\dot\cup C\dot\cup D$, $|C|=4$, $|D|\le4$, $E(L[C])=\varnothing$, $hk\in E(L)$ and $N_L(h)\setminus\{k\}=N_L(k)\setminus\{h\}=C$.
Either $L$ is WKE or $J=L[C\cup D]$ has a matching $M$ and a set $Q$ such that
\[
|M|\ge2,\,|Q|\le|M|+1,\,E(J-Q)\subseteq M,\,|M\cap E(C,D)|\le2.
\]
\end{lemma}
\begin{proof}
\leavevmode\vspace{0.5em} \\
\noindent\textbf{Case 1: $D$ is independent.}\vspace{0.5em} \\*
For contact matching number at least three, use three contacts and $hk$ with $Q=C$.
Otherwise add $h,k$ to a minimum contact cover and match them to two unused leaves.
\leavevmode\vspace{0.5em} \\
\noindent\textbf{Case 2: $|D|=2$.}\vspace{0.5em} \\*
The private edge is present and the contact matching number is one or two.
For two, use $Q=\{h,k\}\cup D$, two contacts and an edge from each of $h,k$ to distinct unused leaves.
For one, use $Q=\{h,k\}$ together with the contact star center, and the private edge with an edge from each of $h,k$ to distinct unused leaves.
\leavevmode\vspace{0.5em} \\
\noindent\textbf{Case 3: $|D|=3$.}\vspace{0.5em} \\*
A private edge $ab$ and a contact $c\ell$ give $M=\{ab,c\ell\}$ and $Q=D$.
Otherwise $L[D]$ is a path with all contacts at its center.
Use the center and $h,k$ for $Q$, with a private edge and an edge from each of $h,k$ to distinct unused leaves for $M$.
\leavevmode\vspace{0.5em} \\
\noindent\textbf{Case 4: $|D|=4$.}\vspace{0.5em} \\*
A private edge $ab$ and disjoint contacts $c\ell,d\ell'$ give $M=\{ab,c\ell,d\ell'\}$ and $Q=D$.
Assume no such choice exists.
If $\{ab,cd\}$ is a private perfect matching, the contacts at $c,d$ have a cover $K$ with $|K|\le1$, so use
\[
M=\{ab,cd\},\,Q=\{a,b\}\cup K.
\]
Every other private edge meets $a$ or $b$.
\leavevmode\vspace{0.5em} \\
The private graphs left are stars and triangles with isolates.
For one edge $ab$ and isolates $c,d$, both isolates contact the same sole leaf $\ell$, so take $M=\{ab,c\ell\}$ and $Q=\{a,b,\ell\}$.
\leavevmode\vspace{0.5em} \\
For a path with center $b$, leaves $a,c$ and isolate $d$, take $M=\{ba,d\ell\}$.
Use $Q=\{b,d\}$ if neither private leaf contacts $C$.
Otherwise their contacts and those of $d$ all use the same leaf $\ell$, so take $Q=\{b,\ell\}$.
\leavevmode\vspace{0.5em} \\
For a three leaf star centered at $b$, the leaf contacts have matching number at most one.
If none occur, $Q=\{h,k,b\}$ gives WKE.
Otherwise choose a contacting leaf $a$, another leaf $c$ and a contact cover $K$ of size one, and take $M=\{bc,a\ell\}$ and $Q=\{b\}\cup K$.
\leavevmode\vspace{0.5em} \\
For a triangle $a,b,c$ and isolate $d$, take $M=\{ab,d\ell\}$.
Use $Q=\{a,b,d\}$ if the triangle has no contact.
Otherwise its contacts and those of $d$ all use the same leaf $\ell$, so take $Q=\{a,b,\ell\}$.
\end{proof}
\begin{lemma}\label{star-isolate-six-ten}
If the selected common graph is $K_{1,4}\cup K_1$ and $|E(H)|$ is maximum, then $G$ has a reduction.
\end{lemma}
\begin{proof}
Let $t$ be the star center, $C$ its four leaves and $s$ the isolate.
Lemma \ref{two-hub-ten} gives three private vertices at each endpoint, inducing $P_3$ or $K_3$ and sharing two neighbors in $C\cup\{s\}$.
A private contact at $t$ would give an edge of codegree $6$ whose common graph has the four edges from the other endpoint to $C$ and a private contact, contradicting maximality of $|E(H)|$.
\leavevmode\vspace{0.5em} \\
So both groups miss $t$, and robustness makes one group contact $s$.
In that group, take a private edge $ab$, pack $uab,ucs$ and pick $ua,ub,uc,ab,cs$.
At $v$, use the same construction with a contacted leaf $p$.
\leavevmode\vspace{0.5em} \\
Each cover has size $5$.
The connected link at $t$ has at most four vertices outside $\{u,v\}\cup C$, all missing $u,v$.
Reduce $t$ if its link is WKE.
\leavevmode\vspace{0.5em} \\
Otherwise Lemma \ref{four-private-ten} gives $M,Q$ with $k=|M|$, using at most two leaves $T\subseteq C$.
Pack the $k$ triangles at $t$ on $M$ and select $M\cup E(\{t\},Q)$, at cost at most $2k+1$.
\leavevmode\vspace{0.5em} \\
Put $V=\{p\}\cap C$.
For $vt,ut,uv$, choose distinct leaves in order from
\[
C\setminus(T\cup V),\,C\setminus T,\,C\setminus V.
\]
Their sizes are at least one, two and three, so the choices exist.
\leavevmode\vspace{0.5em} \\
Pack these three triangles and pick $vt,ut,uv$.
Then $|S|=k+7$ and
\[
|X|\le5+5+(2k+1)+3=2|S|.
\]
The cover meets every triangle through $u,v,t$ and has every packed edge outside these vertices.
\end{proof}
\subsection{No private vertices}
\begin{lemma}\label{d10sparse}
Assume $N(u)=C\cup\{v\}$, $N(v)=C\cup\{u\}$ and $|C|\ge3$.
If $\Delta(H-z)\le2$ for some $z\in V(H)$, where $H=G[C]$, then $\{u,v\}$ is reducible.
\end{lemma}
\begin{proof}
Put $J=H-z$ and write $r=|E(J)|$.
Let $o$ count the odd cycle components of $J$.
Delete one edge from each such cycle and partition the remaining edges into two matchings $M_u,M_v$.
\leavevmode\vspace{0.5em} \\
\noindent\textbf{Case 1: $r>0$.}\vspace{0.5em} \\*
Take
\[
S=\{uxy:xy\in M_u\}\cup\{vxy:xy\in M_v\}\cup\{uvz\},
\]
\[
X=E(J)\cup\{uv,uz,vz\}.
\]
The matchings give a packing $S$, and $X$ covers every triangle through $u,v$ and contains every packed edge outside them.
\leavevmode\vspace{0.5em} \\
Each component with an edge contributes at least one to $r-2o$, so
$|X|=r+3\le2(r-o+1)=2|S|$.
\leavevmode\vspace{0.5em} \\
\noindent\textbf{Case 2: $r=0$.}\vspace{0.5em} \\*
If $zx\in E(H)$, choose $y\in C\setminus\{z,x\}$ and take $S=\{uzx,uvy\}$ and $X=\{uv,zx,uz,vz\}$.
If $E(H)=\varnothing$, take $S=\{uvz\}$ and $X=\{uv\}$.
Both choices cover every triangle through $u,v$ and pick every packed edge outside them.
\end{proof}
\subsection{$K_{3,4}$}
\begin{lemma}\label{d10k34}
Let $\Delta(G)\le10$ and let $uv$ have common graph $K_{3,4}$ with parts $P=\{p_1,p_2,p_3\}$, $Q=\{q_1,q_2,q_3,q_4\}$ and private pairs $A=\{a_1,a_2\}$, $B=\{b_1,b_2\}$, satisfying
\[
P\subseteq\bigcap_{x\in A\cup B}N(x),\,
\{q_1,q_2,q_3\}\subseteq N(a_1),\,
\{q_2,q_3,q_4\}\subseteq N(a_2).
\]
Then $U=\{u,v\}\cup P$ is reducible.
\end{lemma}
\begin{proof}
Every vertex of $U$ has degree $10$ in $J$.
Color $P\cup\{u,b_1,b_2\}$ and $Q\cup\{v,a_1,a_2\}$ differently, and take $X$ as the local edges of equal endpoint colors.
\leavevmode\vspace{0.5em} \\
The edges $uP,vQ,PB$ contribute $13$ to $|X|$, and the contacts from $A$ to $Q$ and the two private edges contribute at most $10$.
Edges between $A,B$ have different endpoint colors, so $|X|\le23$.
\leavevmode\vspace{0.5em} \\
Take this packing with $|S|=12$.
\begin{center}
\begin{minipage}{\textwidth}
\centering\setlength{\tabcolsep}{3pt}
\begin{tabular}{@{}c@{}}
\begin{tabular}{@{}l@{}}
\PackingRow{q_1}{p_1/u,p_2/a_1,p_3/v}\\[0.18em]
\PackingRow{q_2}{p_1/a_1,p_2/a_2,p_3/u}\\[0.18em]
\PackingRow{q_3}{p_1/a_2,p_2/v,u/a_1}\\[0.18em]
\PackingRow{q_4}{p_3/a_2,u/v}\\[0.18em]
\PackingRow{p_1}{v/b_1}\\[0.18em]
\end{tabular}
\end{tabular}

\end{minipage}
\end{center}
The triangles don't share any edges, with $E(S)\setminus E(U,V(G))\subseteq E(A,Q)\subseteq X$.
Every triangle meeting $U$ meets $X$, with $|X|\le23<24=2|S|$.
\end{proof}
\subsection{At most 5 missing edges}
\begin{lemma}\label{d10fiveholes}
Let $\Delta(G)\le10$ and let $uv\in E(G)$ have $9$ common neighbors.
If
\[
|E(\overline{G[N(u)\cap N(v)]})|\le5,
\]
then $N[u]\cup N[v]$ is reducible.
\end{lemma}
\begin{proof}
Put $C=N(u)\cap N(v)$ and $U=\{u,v\}\cup C$, and let $F$ have vertex set $U$ and edge set $E(\overline{G[C]})$, with $r=|E(F)|\le5$.
Both endpoints have degree $10$, so $G[U]=K_{11}-E(F)$.
\leavevmode\vspace{0.5em} \\
Take the below packing with $|S_0|=17$.
The unused edges form the cycle $0,1,2,3,0$.
\begin{center}
\begin{minipage}{\textwidth}
\centering
\begin{tabular}{@{}l@{\hspace{0.65cm}}l@{}}
\PackingRow{0}{2/4,5/6,7/8,9/10}&\PackingRow{2}{5/7,6/9,8/10}\\[0.18em]
\PackingRow{1}{5/10,3/8,4/9,6/7}&\PackingRow{3}{4/5,7/9,6/10}\\[0.18em]
\PackingRow{5}{8/9}&\PackingRow{4}{6/8,7/10}\\[0.18em]
\end{tabular}

\end{minipage}
\end{center}
A triangle avoiding the leave has $0$, $1$ or $2$ cycle vertices, and two cycle vertices must be opposite.
Packing $S_0$ has a triangle of each type, so relabel the cycle and the other $7$ vertices to include any chosen triangle avoiding the leave.
\leavevmode\vspace{0.5em} \\
\noindent\textbf{Case 1: $r\le4$.}\vspace{0.5em} \\*
Choose a leave cycle with $\min\{r,2\}$ edges of $F$, since any two edges lie on a four cycle.
Delete any packed triangle using another edge of $F$.
For $r\le4$ this gives $|S|\ge17-\max\{0,r-2\}$.
\leavevmode\vspace{0.5em} \\
\noindent\textbf{Case 2: $r=5$.}\vspace{0.5em} \\*
A three edge path in $F$ fits in the leave and gives at least $15$ triangles.
Otherwise each nontrivial component of $F$ is a star or a triangle.
If a triangle component occurs, put the other two missing edges in the leave and prescribe that triangle in $S_0$, then delete it.
\leavevmode\vspace{0.5em} \\
If a star has at least four edges, put two of its edges in the leave and prescribe the triangle through two other star edges.
If a star has three edges, put one of its edges and one other missing edge in the leave and prescribe the triangle through its two remaining edges.
If the largest star has two edges, prescribe its two edges in one triangle and put two other missing edges in the leave.
\leavevmode\vspace{0.5em} \\
In each star case delete at most two triangles.
Since $\nu_m(F)\le4$, we get $|S|\ge15$ when $r=5$.
\leavevmode\vspace{0.5em} \\
Partition $U$ into classes of sizes $5,6$ without splitting a component of $F$.
A component of order $5$ or $6$ forms a class, and a largest component of order $3$ or $4$ needs $2$ or $1$ isolates.
The vertices $u,v$ supply both isolates.
\leavevmode\vspace{0.5em} \\
If all nontrivial components have order $2$, at most four occur, so one component and three isolates form the smaller class.
If $F$ is empty, choose any $5$ vertices.
\leavevmode\vspace{0.5em} \\
Select every local edge within a class and every edge from $C$ to a vertex outside $U$.
The local cover has size $25-r$, and the outside edges number at most $\sum_{c\in C}d_F(c)=2r$.
So $|X|\le25+r\le2|S|$.
\leavevmode\vspace{0.5em} \\
Every triangle meeting $U$ meets $X$, and every packing edge lies in $G[U]$.
\end{proof}
\subsection{$K_{2,4}$}
Let $uv\in E(G)$ and put $C=L\cup\{r,s\}$ with $|L|=4$ and $G[C]=K_{2,4}$.
The private pairs are $\{a,b\}$ and $\{c,d\}$, and $k$ counts the edges between them.
\leavevmode\vspace{0.5em} \\
Put $R=\{u,v,r,s,a,b,c,d\}$ and take $U=\{u,v\}\cup L$ in the first three lemmas, with
\[
X=E(G[R])\cup E\bigl(L,V(G)\setminus(R\cup L)\bigr).
\]
A local triangle meeting $U$ has an edge of $G[R]$ since $L$ is independent.
Every other triangle meeting $U$ uses a picked outside edge at a leaf.
\leavevmode\vspace{0.5em} \\
An edge coloring of a graph on $R$ with four colors gives four matchings with no shared edge.
Assign one matching to each leaf $z$ and replace every edge $xy$ of that matching by the triangle $zxy$.
A matching must avoid any private vertex missing its leaf.
\leavevmode\vspace{0.5em} \\
These triangles and any packed triangle on $R$ have every edge outside $U$ in $X$.
\begin{lemma}[Fournier {\cite{Fournier,HiltonZhao}}]\label{fourcolors}
If $\Delta(G)\le4$ and the vertices of degree four induce a forest, then $G$ has an edge coloring with four colors.
\end{lemma}
For $\Delta(G)<4$, adjoin a disjoint $K_{1,4}$ and apply the theorem.
Each color class is a matching.
\begin{lemma}\label{k24twoedges}
Assume $\Delta(G)\le10$, every edge has codegree at most $6$, $ab,cd$ occur, $a,c$ join every leaf and one center, and $b,d$ join no center and miss at most one leaf each.
Then $\{u,v\}\cup L$ is reducible.
\end{lemma}
\begin{proof}
Relabel the centers so that $a$ joins $r$ and write $t$ for the center joined by $c$.
The $13$ fixed edges on $R$ and at most $10$ outside edges give $|X|\le23+k$.
\leavevmode\vspace{0.5em} \\
Assign the four matchings in one of Panels 1 through 8 to distinct leaves.
Use a matching avoiding $b$ for its missed leaf and do the same for $d$.
\newpage
\begin{center}
\begin{tikzpicture}[baseline=-0.5ex]
\foreach \PackingStyle/\PackingNumber [count=\PackingIndex from 0] in {leaf one/1,leaf two/2,leaf three/3,leaf four/4}{%
\draw[\PackingStyle] (1.50*\PackingIndex,0)--(1.50*\PackingIndex+0.48,0);
\node[anchor=west,font=\small] at (1.50*\PackingIndex+0.53,0) {$\ell_{\PackingNumber}$};}
\end{tikzpicture}
\end{center}
\begin{center}
\begin{minipage}{\textwidth}
\centering\setlength{\tabcolsep}{3pt}
\begin{tabular}{@{}cc@{}}
\begin{minipage}[t]{0.47\textwidth}\centering\textbf{Panel 1. $t=r$}\\[0.2em]
\LeafPacking{u/r,a/b,v/d}{u/a,r/v,c/d}{u/b,r/c,v/s}{u/s,v/c,r/a}{}
\end{minipage}&
\begin{minipage}[t]{0.47\textwidth}\centering\textbf{Panel 2. $t=r$}\\[0.2em]
\LeafPacking{a/b,c/d,v/s}{u/a,b/d,r/c}{u/b,a/c,v/d}{u/s,v/c,r/a}{u/v/r}
\end{minipage}\\[0.8em]
\begin{minipage}[t]{0.47\textwidth}\centering\textbf{Panel 3. $t=r$}\\[0.2em]
\LeafPacking{u/r,a/c,v/d}{u/b,r/a,v/c}{v/r,a/b,c/d}{u/a,r/c,b/d}{u/v/s}
\end{minipage}&
\begin{minipage}[t]{0.47\textwidth}\centering\textbf{Panel 4. $t=s$}\\[0.2em]
\LeafPacking{u/a,v/c}{u/b,v/d,r/a,s/c}{a/b,c/d,u/s,v/r}{u/r,v/s}{}
\end{minipage}
\end{tabular}

\end{minipage}
\end{center}
\begin{center}
\begin{minipage}{\textwidth}
\centering\setlength{\tabcolsep}{3pt}
\begin{tabular}{@{}cc@{}}
\begin{minipage}[t]{0.47\textwidth}\centering\textbf{Panel 5. $t=s$}\\[0.2em]
\LeafPacking{u/r,v/c,a/b}{u/b,v/r,a/d,s/c}{u/a,v/s,c/d}{u/s,v/d,r/a}{}
\end{minipage}&
\begin{minipage}[t]{0.47\textwidth}\centering\textbf{Panel 6. $t=s$}\\[0.2em]
\LeafPacking{a/b,c/d,v/s}{u/a,b/c,v/d}{u/b,a/d,s/c}{r/a,v/c,u/s}{u/v/r}
\end{minipage}\\[0.8em]
\begin{minipage}[t]{0.47\textwidth}\centering\textbf{Panel 7. $t=s$}\\[0.2em]
\LeafPacking{r/a,c/d,v/s}{u/a,b/c,v/d}{u/b,a/d,s/c}{a/b,u/s,v/c}{u/v/r}
\end{minipage}&
\begin{minipage}[t]{0.47\textwidth}\centering\textbf{Panel 8. $t=s$}\\[0.2em]
\LeafPacking{u/s,a/c,v/d}{u/b,r/a,v/c}{u/r,v/s,c/d,a/b}{v/r,s/c,b/d,u/a}{}
\end{minipage}
\end{tabular}

\end{minipage}
\end{center}

Panels 2, 3, 6 and 7 also pack the red triangle on $R$.
The four matchings and that triangle don't share any edges.
\leavevmode\vspace{0.5em} \\
\noindent\textbf{Case 1: $k\le1$.}\vspace{0.5em} \\*
Use Panel 1 or 4 according to $t$.
Both give $12$ triangles and allow different missing leaves.
\leavevmode\vspace{0.5em} \\
\noindent\textbf{Case 2: $t=r$, $2\le k\le3$.}\vspace{0.5em} \\*
Add $ad$ to the $\ell_3$ matching of Panel 1 if it occurs, or exchange the two private sides if $bc$ occurs instead.
The remaining case with $t=r$ has cross edges $ac,bd$.
Use Panel 2 if the missing leaves coincide or at most one exists, and Panel 3 otherwise.
\leavevmode\vspace{0.5em} \\
Four cross edges with $t=r$ give $a,c$ $7$ common neighbors, contrary to the codegree bound.
\leavevmode\vspace{0.5em} \\
\noindent\textbf{Case 3: $t=s$, $2\le k\le3$.}\vspace{0.5em} \\*
Add $ac$ to the $\ell_4$ matching of Panel 4 when it occurs.
If $ac$ is absent and $bd$ occurs, add $bd$ to the $\ell_1$ matching when the missing leaves don't differ.
For different missing leaves use Panel 5 if $ad$ occurs or exchange the two private sides and the centers if only $bc$ occurs.
\leavevmode\vspace{0.5em} \\
If neither $ac$ nor $bd$ occurs, the two cross edges are $ad,bc$.
Use Panel 6 when the missing leaves don't differ and Panel 7 otherwise.
These constructions give $13$ triangles.
\leavevmode\vspace{0.5em} \\
\noindent\textbf{Case 4: $t=s$, $k=4$.}\vspace{0.5em} \\*
Add $bd$ to the $\ell_1$ matching and $ac$ to the $\ell_4$ matching of Panel 4 when the missing leaves don't differ.
Otherwise use Panel 8.
Both give $14$ triangles.
\leavevmode\vspace{0.5em} \\
Each missed private vertex has an avoiding matching, with distinct matchings for different missed leaves.
We have $2|S|\ge23+k\ge|X|$, so $U$ is reducible.
\end{proof}
In Forms 1 through 3, $c,d$ join $v$ and every leaf.
The center contacts are $rc,sc$ in Form 1 and $rc,sd$ in Forms 2 and 3, and $cd$ occurs only in Form 3.
\begin{lemma}\label{k24oneedge}
Assume $\Delta(G)\le10$ and
\[
ab\in E(G),\ N(a)\cap C=L\cup\{r\},\ N(b)\cap\{r,s\}=\varnothing,\ |L\setminus N(b)|\le1.
\]
If $c,d$ join every leaf and have one of Forms 1 through 3, then $\{u,v\}\cup L$ is reducible.

\end{lemma}
\begin{proof}
Put $\beta=0$ in Forms 1 and 2 and $\beta=1$ in Form 3.
At most $9$ outside edges occur, so $|X|\le22+\beta+k$.
\leavevmode\vspace{0.5em} \\
Pack $uvr$ and delete its three edges from the graph on $R$.
Keep
\[
q=\begin{cases}\lceil k/2\rceil&\beta=0,\\\lfloor k/2\rfloor&\beta=1\end{cases}
\]
cross edges, with any edge when $q=1$.
\leavevmode\vspace{0.5em} \\
For $q=2$, use a pair meeting each of $a,b$ once and $c$ at most once in Form 1, a pair meeting each of $a,b$ once in Form 2, and $ac,bd$ in Form 3.
In Forms 1 and 2, $k\ge3$ gives such a pair, and Form 3 has $k=4$.
\leavevmode\vspace{0.5em} \\
Before these additions, $u,v,s,a$ have degree three and $r,b$ have degree two.
The degrees at $c,d$ are $3,1$ in Form 1, $2,2$ in Form 2 and $3,3$ in Form 3.
\leavevmode\vspace{0.5em} \\
Each choice gives $\Delta\le4$, $d(b)\le3$ and a forest on the vertices of degree four.
The forest is $acd$ for the two edge choice in Form 3 and has at most two vertices otherwise.
\leavevmode\vspace{0.5em} \\
Lemma \ref{fourcolors} gives four color matchings.
A color class avoids $b$ since its degree is at most three, so assign that class to its missed leaf and assign the other classes to the other leaves.
The leaf triangles together with $uvr$ form a packing of size $11+\beta+q$.
\leavevmode\vspace{0.5em} \\
For $\beta=0$ we have $k\le2q$, and for $\beta=1$ we have $k-1\le2q$.
So $|X|\le22+\beta+k\le2|S|$.
\end{proof}
\begin{lemma}\label{k24fullleaf}
Suppose $\Delta(G)\le10$, every private vertex joins every leaf, and each side has one of the three forms in Lemma \ref{k24oneedge}.
Then $\{u,v\}\cup L$ is reducible.
\end{lemma}
\begin{proof}
Let $\epsilon$ count the private edges.
Each leaf has $8$ neighbors in $R$, so $|X|\le21+\epsilon+k$.
Pack $uvr$ and get rid of its three edges from the graph on $R$.
\leavevmode\vspace{0.5em} \\
In the remaining graph $u,v$ have degree three, $r$ has degree two and $s$ has degree four.
The degree pairs on a private side are $3,1$ or $2,2$ or $3,3$ according to its form.
Keep
$$q=\max\{0,\lceil(k-1-\epsilon)/2\rceil\}$$
cross edges.
\leavevmode\vspace{0.5em} \\
For $q=1$, choose an edge whose endpoints don't both have initial degree three and both join $s$.
At most one cross edge has that forbidden form and $q=1$ needs at least two present cross edges.
For $q=2$, we have $\epsilon=0,k=4$, so keep $ad,bc$.
\leavevmode\vspace{0.5em} \\
The chosen graph has $\Delta\le4$.
For $q=1$ the vertices of degree four induce a forest because the chosen edge avoids the only possible triangle with $s$, and for $q=2$ they induce a star centered at $s$.
\leavevmode\vspace{0.5em} \\
Color its edges with four colors by Lemma \ref{fourcolors} and assign the four color matchings to the four leaves.
Their leaf triangles together with $uvr$ give
\[
|S|=11+\epsilon+q,\,|X|\le21+\epsilon+k\le22+2\epsilon+2q=2|S|.
\]
\end{proof}
\begin{lemma}\label{k24twoleaf}
Assume $\Delta(G)\le10$, label $L=\{0,1,2,3\}$, and suppose
\[
N(a)\cap L=N(b)\cap L=\{0,1\},\,N(c)\cap L=N(d)\cap L=\{2,3\},
\]
with center contacts $sa,rb,sc,rd$.
Then $\{u,v,r,s\}$ is reducible for every graph on $\{a,b,c,d\}$.
\end{lemma}
\begin{proof}
Take $U=\{u,v,r,s\}$ and use the below packing and cover.
Add every edge from $r,s$ to a vertex outside the local graph to the blue cover.
$d_J(r),d_J(s)\ge8$, so $|X|\le16+4=20=2|S|$.
\begin{center}
\begin{minipage}{\textwidth}
\centering\setlength{\tabcolsep}{3pt}
\begin{tabular}{@{}cc@{}}
\begin{tabular}{@{}l@{}}
\PackingRow{0}{r/b,s/a,u/v}\\[0.18em]
\PackingRow{1}{s/v,u/b}\\[0.18em]
\PackingRow{2}{r/u,s/c,v/d}\\[0.18em]
\PackingRow{3}{r/v,s/u}\\[0.18em]
\end{tabular}&
\begin{tikzpicture}[baseline=-0.5ex]
\foreach \GraphVertex/\GraphX/\GraphY in {0/0.00000/1.40000,1/0.70000/1.21244,2/1.21244/0.70000,3/1.40000/0.00000,4/1.21244/-0.70000,5/0.70000/-1.21244,6/0.00000/-1.40000,7/-0.70000/-1.21244,8/-1.21244/-0.70000,9/-1.40000/-0.00000,10/-1.21244/0.70000,11/-0.70000/1.21244}{\coordinate (p\GraphVertex) at (\GraphX,\GraphY);}
\foreach \GraphFrom/\GraphTo in {4/8,4/9,5/9,6/10,6/11,6/1,7/10,8/3,8/0,11/2,11/1,2/0,2/1,3/0,3/1,0/1}{\draw[blue!75!black,line width=0.7pt] (p\GraphFrom)--(p\GraphTo);}
\foreach \GraphVertex/\GraphLabel in {0/{u},1/{v},2/{r},3/{s},4/{0},5/{1},6/{2},7/{3},8/{a},9/{b},10/{c},11/{d}}{\node[proof vertex] at (p\GraphVertex) {$\GraphLabel$};}
\end{tikzpicture}
\end{tabular}

\end{minipage}
\end{center}
The 10 triangles don't share any edges, and every packed edge outside $U$ is blue.
After deleting $X$, the remaining neighborhoods satisfy
$$N_{G-X}(u),N_{G-X}(r)\subseteq L\cup\{b\}$$
$$N_{G-X}(v)\subseteq\{0,1,3,c\}$$
$$N_{G-X}(s)\subseteq L\cup\{c\}.$$
Each set on the RHS is independent in $G-X$, so no triangle meets $U$.
\end{proof}
\subsection{More reductions}
For $c=5$, edge deletions must obey the choice maximizing $|E(H)|$.
For $c\ge7$, the bound $|A|+|B|\le18-2c$ limits the extra cover cost at the private vertices.
\leavevmode\vspace{0.5em} \\
\noindent\textbf{Case 1: $c=5$.}\vspace{0.5em} \\*
\label{c5-ten}
Say four common neighbors of an edge induce more edges than $H$.
A further common neighbor contradicts the choice maximizing $|E(H)|$ or the codegree bound, so these four are all its common neighbors.
Cover every triangle on the edge before deleting it without selecting it.
\leavevmode\vspace{0.5em} \\
\noindent\textbf{Subcase 1a: $H$ is empty.}\vspace{0.5em} \\*
Apply the deletion rule to each private triangle edge in the below block.
The diagram on $h,a,b,c,x,y$ gives $|S|=4$, $|X|=8$ for $K_6-xy$.
\begin{center}
\begin{minipage}{\textwidth}
\centering\setlength{\tabcolsep}{3pt}
\begin{tabular}{@{}cc@{}}
\begin{tabular}{@{}l@{}}
\PackingRow{h}{a/x,b/y}\\[0.18em]
\PackingRow{a}{c/y}\\[0.18em]
\PackingRow{b}{c/x}\\[0.18em]
\end{tabular}&
\begin{tikzpicture}[baseline=-0.5ex]
\foreach \GraphVertex/\GraphX/\GraphY in {0/0.00000/0.92000,1/0.79674/0.46000,2/0.79674/-0.46000,3/0.00000/-0.92000,4/-0.79674/-0.46000,5/-0.79674/0.46000}{\coordinate (p\GraphVertex) at (\GraphX,\GraphY);}
\foreach \GraphFrom/\GraphTo in {1/2,1/4,1/5,2/4,2/5,3/0,3/4,3/5}{\draw[blue!75!black,line width=0.7pt] (p\GraphFrom)--(p\GraphTo);}
\foreach \GraphVertex/\GraphLabel in {0/{h},1/{a},2/{b},3/{c},4/{x},5/{y}}{\node[proof vertex] at (p\GraphVertex) {$\GraphLabel$};}
\end{tikzpicture}
\end{tabular}

\end{minipage}
\end{center}
\leavevmode\vspace{0.5em} \\
Assume a private triangle has complete contacts to a pair $\{x,y\}\subset C$ and no other common contacts.
The side may also contain a private isolate.
\leavevmode\vspace{0.5em} \\
Two blocks and a triangle using $uv$ give $|S|=9$, $|X|=18$ when one edge from an endpoint to its private isolate suffices.
With two isolates, add a contact triangle at an isolate to get $|S|=10$, $|X|=20$.
\vspace{-0.5em} \\ In the link at $u$, label the private triangle $abc$, its isolate $d$ and the other endpoint $v$.
If $d$ has one contact $dq$ outside $\{x,y\}$, take $M=\{bc,ax,dq,v\ell\}$, $Q=\{v,x,y,a\}$ with $\ell\notin\{x,y,q\}$.
If $d$ has no such contact, use $M=\{bc,dx,ay,v\ell\}$ and the same $Q$ after exchanging $x,y$ if needed.
\leavevmode\vspace{0.5em} \\
Both give WKE, so $d$ has at least two contacts outside $\{x,y\}$.
Choose the triangle using $uv$ outside both private contact pairs, then choose the isolate contact outside that leaf and its own pair.
The resulting triangles don't share any edges.
\leavevmode\vspace{0.5em} \\
\noindent\textbf{Subcase 1b: A private $K_4-e$ or $K_4$.}\vspace{0.5em} \\*
Name the private vertices $t,b,p,q$ with only $pq$ absent in the first form.
Assume the common contacts are $tw,pw,qw$.
Get rid of $vt$ and take the below packing of size four.
\begin{center}
\begin{minipage}{\textwidth}
\centering\setlength{\tabcolsep}{3pt}
\begin{tabular}{@{}cc@{}}
\begin{tabular}{@{}l@{}}
\PackingRow{v}{w/p,b/q}\\[0.18em]
\PackingRow{t}{w/q,b/p}\\[0.18em]
\end{tabular}&
\begin{tikzpicture}[baseline=-0.5ex]
\foreach \GraphVertex/\GraphX/\GraphY in {0/0.00000/0.92000,1/0.79674/0.46000,2/0.79674/-0.46000,3/0.00000/-0.92000,4/-0.79674/-0.46000,5/-0.79674/0.46000}{\coordinate (p\GraphVertex) at (\GraphX,\GraphY);}
\foreach \GraphFrom/\GraphTo in {3/4,3/5,3/2,4/2,4/1,5/2,5/1,2/1}{\draw[blue!75!black,line width=0.7pt] (p\GraphFrom)--(p\GraphTo);}
\foreach \GraphVertex/\GraphLabel in {0/{v},1/{w},2/{t},3/{b},4/{p},5/{q}}{\node[proof vertex] at (p\GraphVertex) {$\GraphLabel$};}
\end{tikzpicture}
\end{tabular}

\end{minipage}
\end{center}
\leavevmode\vspace{0.5em} \\
Pick the $8$ edges not incident to the endpoint to cover its triangles and every packed edge outside it.
For $K_4$, three private vertices join $w$ and the fourth misses $C$, giving $K_6-wb$.
Use the $K_6$ minus an edge construction.
\leavevmode\vspace{0.5em} \\
Combine either block with the triangle and isolate block and a triangle using $uv$ outside their common contacts, and pick the edge from the endpoint to the isolate.
We get $|S|=9$, $|X|=18$, and every unselected packed private edge satisfies the deletion rule.
\leavevmode\vspace{0.5em} \\
\noindent\textbf{Subcase 1c: $H$ is a star.}\vspace{0.5em} \\*
Write $z$ for the star center and label the leaves $0,1,2,3$.
The vertices $u,v,z$ form a triangle with these four independent common neighbors.
The codegree bound excludes contacts from a private group to another vertex of $\{u,v,z\}$, and the WKE constructions for at most two private vertices give $d(u),d(v),d(z)\in\{9,10\}$.
\leavevmode\vspace{0.5em} \\
Choose three distinct leaves and pack one triangle on each of $uv,uz,vz$.
For $h\in\{u,v,z\}$, choose a matching $M_h$ on its private edges and leaf contacts, avoiding its two used leaves, and a vertex cover $Q_h$ after deleting $M_h$.
If $|Q_h|\le|M_h|+1$, pick $uv,uz,vz$, the matching edges and every edge from $h$ to $Q_h$ to get
\[
|S|=3+\sum_h|M_h|,\,|X|\le6+2\sum_h|M_h|=2|S|.
\]
Every packed edge outside $\{u,v,z\}$ belongs to $X$.
\leavevmode\vspace{0.5em} \\
For $U=\{u,v,z\}$, take one of the three below constructions.
At each $h\in\{u,v,z\}$, the labels $a,\ldots,l$ in its matching or blue edge diagram form its private set.
Its neighborhood consists of that set, $\{u,v,z\}\setminus\{h\}$ and $\{0,1,2,3\}$.
\leavevmode\vspace{0.5em} \\
Require the $10$ below triangles and assume the crossed edges are absent. Allow all other contacts and private edges.
Take $X$ as the blue edges together with $E(S)\setminus E(U,V(G))$.
Then $|X|=20=2|S|$ and $X$ meets every triangle at $u,v,z$.
\begin{center}
\begin{minipage}{\textwidth}
\centering\setlength{\tabcolsep}{3pt}
\leavevmode\vspace{0.5em} \\
\textbf{Construction 1}\vspace{0.5em} \\
\begin{tabular}{@{}l@{}}
\PackingRow{u}{v/3,z/0,1/c,a/b}\\[0.18em]
\PackingRow{v}{z/1,d/g,e/f}\\[0.18em]
\PackingRow{z}{2/j,3/i,h/k}\\[0.18em]
\end{tabular}
\\[0.3em]
\begin{tabular}{@{}ccccc@{}}
\begin{tikzpicture}[baseline=-0.5ex]
\coordinate (p0) at (0,0);
\coordinate (p1) at (0.00000,0.61000);
\draw[blue!75!black,line width=0.7pt] (p0)--(p1);
\coordinate (p2) at (0.52828,-0.30500);
\draw[blue!75!black,line width=0.7pt] (p0)--(p2);
\coordinate (p3) at (-0.52828,-0.30500);
\draw[blue!75!black,line width=0.7pt] (p0)--(p3);
\foreach \GraphVertex/\GraphLabel in {0/{u},1/{a},2/{b},3/{c}}{\node[proof vertex] at (p\GraphVertex) {$\GraphLabel$};}
\end{tikzpicture}&
\begin{tikzpicture}[baseline=-0.5ex]
\coordinate (p0) at (0,0);
\coordinate (p1) at (0.00000,0.61000);
\draw[blue!75!black,line width=0.7pt] (p0)--(p1);
\coordinate (p2) at (0.58014,0.18850);
\draw[blue!75!black,line width=0.7pt] (p0)--(p2);
\coordinate (p3) at (0.35855,-0.49350);
\draw[blue!75!black,line width=0.7pt] (p0)--(p3);
\coordinate (p4) at (-0.35855,-0.49350);
\draw[blue!75!black,line width=0.7pt] (p0)--(p4);
\coordinate (p5) at (-0.58014,0.18850);
\draw[blue!75!black,line width=0.7pt] (p0)--(p5);
\foreach \GraphVertex/\GraphLabel in {0/{v},1/{0},2/{1},3/{2},4/{3},5/{f}}{\node[proof vertex] at (p\GraphVertex) {$\GraphLabel$};}
\end{tikzpicture}&
\begin{tikzpicture}[baseline=-0.5ex]
\coordinate (p0) at (0,0);
\coordinate (p1) at (0.00000,0.61000);
\draw[blue!75!black,line width=0.7pt] (p0)--(p1);
\coordinate (p2) at (0.61000,0.00000);
\draw[blue!75!black,line width=0.7pt] (p0)--(p2);
\coordinate (p3) at (0.00000,-0.61000);
\draw[blue!75!black,line width=0.7pt] (p0)--(p3);
\coordinate (p4) at (-0.61000,-0.00000);
\draw[blue!75!black,line width=0.7pt] (p0)--(p4);
\foreach \GraphVertex/\GraphLabel in {0/{z},1/{h},2/{i},3/{j},4/{k}}{\node[proof vertex] at (p\GraphVertex) {$\GraphLabel$};}
\end{tikzpicture}&
{\color{blue!75!black}$uz\in X$}&
\begin{tikzpicture}[baseline=-0.5ex]
\foreach \GraphVertex/\GraphX/\GraphY in {0/0.00000/0.66000,1/0.57158/-0.33000,2/-0.57158/-0.33000}{\coordinate (p\GraphVertex) at (\GraphX,\GraphY);}
\draw[gray,densely dashed,line width=0.65pt] (p0)--(p1);
\path (p0)--(p1) node[midway,text=red!75!black,font=\scriptsize] {$\times$};
\draw[gray,densely dashed,line width=0.65pt] (p1)--(p2);
\path (p1)--(p2) node[midway,text=red!75!black,font=\scriptsize] {$\times$};
\foreach \GraphVertex/\GraphLabel in {0/{d},1/{e},2/{g}}{\node[proof vertex] at (p\GraphVertex) {$\GraphLabel$};}
\end{tikzpicture}\\
\end{tabular}\\[0.65em]
\leavevmode\vspace{0.5em} \\
\textbf{Construction 2}\vspace{0.5em} \\
\begin{tabular}{@{}l@{}}
\PackingRow{u}{v/3,z/1,0/d,a/b}\\[0.18em]
\PackingRow{v}{z/2,0/e,1/h,f/g}\\[0.18em]
\PackingRow{z}{i/l,j/k}\\[0.18em]
\end{tabular}
\\[0.3em]
\begin{tabular}{@{}ccccc@{}}
\begin{tikzpicture}[baseline=-0.5ex]
\coordinate (p0) at (0,0);
\coordinate (p1) at (0.00000,0.61000);
\draw[blue!75!black,line width=0.7pt] (p0)--(p1);
\coordinate (p2) at (0.61000,0.00000);
\draw[blue!75!black,line width=0.7pt] (p0)--(p2);
\coordinate (p3) at (0.00000,-0.61000);
\draw[blue!75!black,line width=0.7pt] (p0)--(p3);
\coordinate (p4) at (-0.61000,-0.00000);
\draw[blue!75!black,line width=0.7pt] (p0)--(p4);
\foreach \GraphVertex/\GraphLabel in {0/{u},1/{a},2/{b},3/{c},4/{d}}{\node[proof vertex] at (p\GraphVertex) {$\GraphLabel$};}
\end{tikzpicture}&
\begin{tikzpicture}[baseline=-0.5ex]
\coordinate (p0) at (0,0);
\coordinate (p1) at (0.00000,0.61000);
\draw[blue!75!black,line width=0.7pt] (p0)--(p1);
\coordinate (p2) at (0.61000,0.00000);
\draw[blue!75!black,line width=0.7pt] (p0)--(p2);
\coordinate (p3) at (0.00000,-0.61000);
\draw[blue!75!black,line width=0.7pt] (p0)--(p3);
\coordinate (p4) at (-0.61000,-0.00000);
\draw[blue!75!black,line width=0.7pt] (p0)--(p4);
\foreach \GraphVertex/\GraphLabel in {0/{v},1/{e},2/{f},3/{g},4/{h}}{\node[proof vertex] at (p\GraphVertex) {$\GraphLabel$};}
\end{tikzpicture}&
\begin{tikzpicture}[baseline=-0.5ex]
\coordinate (p0) at (0,0);
\coordinate (p1) at (0.00000,0.61000);
\draw[blue!75!black,line width=0.7pt] (p0)--(p1);
\coordinate (p2) at (0.61000,0.00000);
\draw[blue!75!black,line width=0.7pt] (p0)--(p2);
\coordinate (p3) at (0.00000,-0.61000);
\draw[blue!75!black,line width=0.7pt] (p0)--(p3);
\coordinate (p4) at (-0.61000,-0.00000);
\draw[blue!75!black,line width=0.7pt] (p0)--(p4);
\foreach \GraphVertex/\GraphLabel in {0/{z},1/{0},2/{1},3/{2},4/{3}}{\node[proof vertex] at (p\GraphVertex) {$\GraphLabel$};}
\end{tikzpicture}&
{\color{blue!75!black}$uv\in X$}&
\begin{tikzpicture}[baseline=-0.5ex]
\foreach \GraphVertex/\GraphX/\GraphY in {0/0.00000/0.66000,1/0.66000/0.00000,2/0.00000/-0.66000,3/-0.66000/-0.00000}{\coordinate (p\GraphVertex) at (\GraphX,\GraphY);}
\draw[gray,densely dashed,line width=0.65pt] (p0)--(p1);
\path (p0)--(p1) node[midway,text=red!75!black,font=\scriptsize] {$\times$};
\draw[gray,densely dashed,line width=0.65pt] (p0)--(p2);
\path (p0)--(p2) node[midway,text=red!75!black,font=\scriptsize] {$\times$};
\draw[gray,densely dashed,line width=0.65pt] (p1)--(p3);
\path (p1)--(p3) node[midway,text=red!75!black,font=\scriptsize] {$\times$};
\draw[gray,densely dashed,line width=0.65pt] (p2)--(p3);
\path (p2)--(p3) node[midway,text=red!75!black,font=\scriptsize] {$\times$};
\foreach \GraphVertex/\GraphLabel in {0/{i},1/{j},2/{k},3/{l}}{\node[proof vertex] at (p\GraphVertex) {$\GraphLabel$};}
\end{tikzpicture}\\
\end{tabular}\\[0.65em]
\leavevmode\vspace{0.5em} \\
\textbf{Construction 3}\vspace{0.5em} \\
\begin{tabular}{@{}l@{}}
\PackingRow{u}{v/0,z/1,2/b,3/c,a/d}\\[0.18em]
\PackingRow{v}{z/2,3/g,e/h}\\[0.18em]
\PackingRow{z}{i/l,j/k}\\[0.18em]
\end{tabular}
\\[0.3em]
\begin{tabular}{@{}ccccc@{}}
\begin{tikzpicture}[baseline=-0.5ex]
\coordinate (p0) at (0,0);
\coordinate (p1) at (0.00000,0.61000);
\draw[blue!75!black,line width=0.7pt] (p0)--(p1);
\coordinate (p2) at (0.61000,0.00000);
\draw[blue!75!black,line width=0.7pt] (p0)--(p2);
\coordinate (p3) at (0.00000,-0.61000);
\draw[blue!75!black,line width=0.7pt] (p0)--(p3);
\coordinate (p4) at (-0.61000,-0.00000);
\draw[blue!75!black,line width=0.7pt] (p0)--(p4);
\foreach \GraphVertex/\GraphLabel in {0/{u},1/{a},2/{b},3/{c},4/{d}}{\node[proof vertex] at (p\GraphVertex) {$\GraphLabel$};}
\end{tikzpicture}&
\begin{tikzpicture}[baseline=-0.5ex]
\coordinate (p0) at (0,0);
\coordinate (p1) at (0.00000,0.61000);
\draw[blue!75!black,line width=0.7pt] (p0)--(p1);
\coordinate (p2) at (0.61000,0.00000);
\draw[blue!75!black,line width=0.7pt] (p0)--(p2);
\coordinate (p3) at (0.00000,-0.61000);
\draw[blue!75!black,line width=0.7pt] (p0)--(p3);
\coordinate (p4) at (-0.61000,-0.00000);
\draw[blue!75!black,line width=0.7pt] (p0)--(p4);
\foreach \GraphVertex/\GraphLabel in {0/{v},1/{e},2/{f},3/{g},4/{h}}{\node[proof vertex] at (p\GraphVertex) {$\GraphLabel$};}
\end{tikzpicture}&
\begin{tikzpicture}[baseline=-0.5ex]
\coordinate (p0) at (0,0);
\coordinate (p1) at (0.00000,0.61000);
\draw[blue!75!black,line width=0.7pt] (p0)--(p1);
\coordinate (p2) at (0.61000,0.00000);
\draw[blue!75!black,line width=0.7pt] (p0)--(p2);
\coordinate (p3) at (0.00000,-0.61000);
\draw[blue!75!black,line width=0.7pt] (p0)--(p3);
\coordinate (p4) at (-0.61000,-0.00000);
\draw[blue!75!black,line width=0.7pt] (p0)--(p4);
\foreach \GraphVertex/\GraphLabel in {0/{z},1/{0},2/{1},3/{2},4/{3}}{\node[proof vertex] at (p\GraphVertex) {$\GraphLabel$};}
\end{tikzpicture}&
{\color{blue!75!black}$uv\in X$}&
\begin{tikzpicture}[baseline=-0.5ex]
\foreach \GraphVertex/\GraphX/\GraphY in {0/0.00000/0.66000,1/0.66000/0.00000,2/0.00000/-0.66000,3/-0.66000/-0.00000}{\coordinate (p\GraphVertex) at (\GraphX,\GraphY);}
\draw[gray,densely dashed,line width=0.65pt] (p0)--(p1);
\path (p0)--(p1) node[midway,text=red!75!black,font=\scriptsize] {$\times$};
\draw[gray,densely dashed,line width=0.65pt] (p0)--(p2);
\path (p0)--(p2) node[midway,text=red!75!black,font=\scriptsize] {$\times$};
\draw[gray,densely dashed,line width=0.65pt] (p1)--(p3);
\path (p1)--(p3) node[midway,text=red!75!black,font=\scriptsize] {$\times$};
\draw[gray,densely dashed,line width=0.65pt] (p2)--(p3);
\path (p2)--(p3) node[midway,text=red!75!black,font=\scriptsize] {$\times$};
\foreach \GraphVertex/\GraphLabel in {0/{i},1/{j},2/{k},3/{l}}{\node[proof vertex] at (p\GraphVertex) {$\GraphLabel$};}
\end{tikzpicture}\\
\end{tabular}\\[0.65em]

\end{minipage}
\end{center}
\leavevmode\vspace{0.5em} \\
For $c\ge7$, start with a reduction of $\{u,v\}$ inside $\{u,v\}\cup C$ and put $m=|A|+|B|$.
Select $E(\{u\},A)\cup E(\{v\},B)$ if $|X|+m\le2|S|$.
Add $t$ triangles at $u$ or $v$ using private edges if
\[
|X|+m-t\le2|S|,
\]
since the cover grows by $m+t$ and twice the packing grows by $2t$.
\leavevmode\vspace{0.5em} \\
Require the added triangles to share no edge with each other or with the first packing.
\leavevmode\vspace{0.5em} \\
\noindent\textbf{Case 2: $c=7$.}\vspace{0.5em} \\*
\label{c7-ten}
Since $m\le4$, constructions with $|X|+4\le2|S|$ extend to every private side.
\leavevmode\vspace{0.5em} \\
\noindent\textbf{Case 3: $c=8$.}\vspace{0.5em} \\*
\label{c8-ten}
Since $m\le2$, constructions with $|X|+2\le2|S|$ extend to every private side.
\leavevmode\vspace{0.5em} \\
\noindent\textbf{Case 4: $c=9$.}\vspace{0.5em} \\*
\label{c9-ten}
Both endpoints have degree $10$ and no private vertices.
A pair reduction on $8$ common vertices with $|X|+2\le2|S|$ extends to a common set of size $9$ by picking $uz,vz$ at the added vertex $z$.
These edges cover the new triangles at $u,v$ and keep $S$ unchanged.
\section{Verification}\label{finite-completion}\label{verification-ten}
The Github files \cite{Supplement} check all $40$ endpoint degree pairs from Lemma \ref{endpoint-degrees-ten}.
For each pair they check common graphs, private graphs and contacts with the codegree bound, connected non WKE links and the edge selection rules.
\leavevmode\vspace{0.5em} \\
The files give the allowed and needed local edges for each reduction.
Lemma \ref{interval-ten} allows the other edges at a pair, and the bounds on the number of outside edges apply when the reduction removes common vertices.
\begin{proposition}[{\cite{Supplement}}]\label{forty-ten}
For each $c\in\{4,5,6,7,8,9\}$, every selected edge domain in Lemma \ref{endpoint-degrees-ten} has a vertex or edge reduction.
\end{proposition}
\begin{proof}[Proof of Theorem \ref{main}]
Take the minimal counterexample $G$ fixed in Section \ref{reductions-ten}.
Subsection \ref{subcubic-ten} excludes maximum codegree at most $3$. So the maximum codegree is in $\{4,5,6,7,8,9\}$. Choose the edge by Subsection \ref{choose-ten} giving the $40$ endpoint degree pairs in Lemma \ref{endpoint-degrees-ten}. Proposition \ref{forty-ten} gives a reduction, and Lemma \ref{combine-ten} or its edge version gives Tuza's inequality for $G$, contradicting minimality. We are done.
\end{proof}
\section{Conclusion}
The next natural step is to show Tuza's conjecture for $\Delta\le11$. The same reduction framework applies, but we need a way to handle subcubic links on $11$ vertices, more disconnected endpoint links, and bigger cover costs. We don't know if extending the computation would settle these cases.
\vspace{0.5em} \\
We might be able to push the $\Delta$ bound more with some effort and many more pages, but this method doesn't indicate any way to make the argument uniform.
\leavevmode\vspace{0.5em} \\
Another question is if we can prove the degree $10$ case without finite verification. The edge coloring constructions suggest looking for reductions that cover whole families of common neighborhoods.
The packings alone don't give such reductions, since the cover also needs to account for edges outside the local graph. It would also be interesting to find reductions valid without a $\Delta$ bound that improve Puleo's $\text{mad}$ bound \cite{Puleo}.
\section*{Note on AI usage}
AI wrote the code in the Github repository \cite{Supplement} and assisted in the latter half of the proof of Lemma \ref{thin-ten-ten}.
All other mathematics was done without AI assistance.

\end{document}